\RequirePackage{fix-cm}
\documentclass[smallextended]{svjour3}       % onecolumn (second format)
\smartqed  % flush right qed marks, e.g. at end of proof
\usepackage{graphicx}
\usepackage{mathptmx}      % use Times fonts if available on your TeX system
\usepackage{amsmath}
\usepackage{amssymb}
\usepackage{bm}
\usepackage[shortlabels]{enumitem}
\usepackage[a]{esvect}
\usepackage{hyperref}
\usepackage[numbers,sort&compress]{natbib}
\usepackage{xcolor}

\DeclareMathOperator{\Real}{Re}
\DeclareMathOperator{\Id}{Id}
\DeclareMathOperator{\ran}{ran}
\spnewtheorem*{remark*}{Remark}{\it}{\rm}
\allowdisplaybreaks

\journalname{Research in the Mathematical Sciences}
\begin{document}

\title{Delay-coordinate maps, coherence, and approximate spectra of evolution operators\thanks{This paper is dedicated to Andrew Majda on the occasion of his 70th birthday.} 
}


\author{Dimitrios Giannakis}


\institute{D. Giannakis \at
              Department of Mathematics and Center for Atmosphere Ocean Science \\
              Courant Institute of Mathematical Sciences \\
              New York University\\
              251 Mercer St, New York, NY 10012, USA\\
              \email{dimitris@cims.nyu.edu}           %  \\
           }

\date{Received: date / Accepted: date}

\maketitle

\begin{abstract}
    The problem of data-driven identification of coherent observables of measure-preserving, ergodic dynamical systems is studied using kernel integral operator techniques. An approach is proposed whereby complex-valued observables with approximately cyclical behavior are constructed from a pair of eigenfunctions of integral operators built from delay-coordinate mapped data. It is shown that these observables are $\epsilon$-approximate eigenfunctions of the Koopman evolution operator of the system, with a bound $\epsilon$ controlled by the length of the delay-embedding window,  the evolution time, and appropriate spectral gap parameters. In particular, $ \epsilon$ can be made arbitrarily small as the embedding window increases so long as the corresponding eigenvalues remain sufficiently isolated in the spectrum of the integral operator. It is also shown that the time-autocorrelation functions of such observables are $\epsilon$-approximate Koopman eigenvalues, exhibiting a well-defined characteristic oscillatory frequency (estimated using the Koopman generator) and a slowly-decaying modulating envelope. The results hold for measure-preserving, ergodic dynamical systems of arbitrary spectral character, including mixing systems with continuous spectrum and no non-constant Koopman eigenfunctions in $L^2$. Numerical examples reveal a coherent observable of the Lorenz 63 system whose autocorrelation function remains above 0.5 in modulus over approximately 10 Lyapunov timescales.  
\end{abstract}

\section{\label{secIntro}Introduction}

\subsection{Background}
In the papers \cite{GiannakisMajda11c,GiannakisMajda12a,GiannakisMajda13}, A.\ J.\ Majda and the author proposed a decomposition technique for multivariate time series, called nonlinear Laplacian spectral analysis (NLSA), combining aspects of delay-coordinate maps of dynamical systems with kernel methods for machine learning. NLSA treats the sampled time series as an observable of a dynamical system, and embeds it into a higher-dimensional space using Takens' method of delays \cite{PackardEtAl80,SauerEtAl91,Takens81}. Nonlinear features (principal components) are then extracted as eigenvectors of a normalized kernel matrix constructed from the delay-embedded data, adopting the perspective of geometrical learning techniques such as Laplacian eigenmaps \cite{BelkinNiyogi03} and diffusion maps \cite{CoifmanLafon06}. One of the principal empirical findings in \cite{GiannakisMajda11c,GiannakisMajda12a,GiannakisMajda13} was that the leading modes in the NLSA decomposition exhibit a coherent temporal evolution, capturing distinct timescales from multiscale input data. Examples include systems of ordinary differential equations with metastable regime behavior \cite{GiannakisMajda12a}, as well as simulated and observed climate data \cite{GiannakisMajda12b,SzekelyEtAl16a}. 

Meanwhile, in independent work \cite{BerryEtAl13}, Berry et al.\ developed an analysis technique called diffusion-mapped delay coordinates (DMDC) which is based on a related delay-coordinate kernel construction, and gave a theoretical interpretation of the timescale separation capability of the DMDC modes using the Oseledets multiplicative ergodic theorem and Lyapunov metrics of dynamical systems. In particular, they showed that under smoothness and hyperbolicity assumptions on the dynamics, and for an appropriately weighted delay-embedding scheme, as the number of delays increases the leading eigenfunctions recovered through the diffusion maps algorithm vary predominantly along the Oseledets subspace associated with the most stable Lyapunov exponent of the system. They then argued that the evolution of these eigenfunctions, viewed as reduced coordinates for the system state, can be well modeled as a gradient flow driven by a nonautonomous perturbation from the remaining degrees of freedom. In this picture, diffusion maps captures the leading eigenfunctions of the generator of a stochastic process,  exhibiting distinct timescales associated with the corresponding eigenvalues.

Besides DMDC and NLSA, several other feature extraction techniques utilizing delay-coordinate maps have been proposed, including early methods such as singular spectrum analysis (SSA) \cite{BroomheadKing86,VautardGhil89} and more recent techniques where connections with operator-theoretic ergodic theory have been emphasized \cite{MezicBanaszuk04,ArbabiMezic17,BruntonEtAl17}. While NLSA and DMDC differ from these methods in the use of nonlinear kernels (which allow recovery of nonlinear features), the general consensus stemming from this body of literature  is that incorporating delays in feature extraction methodologies facilitates the recovery of dynamically relevant, coherent patterns. Note that this property is distinct from topological state space reconstruction from partial observations (which was the original purpose of delay-coordinate maps \cite{PackardEtAl80}), and can be beneficial even under fully observed scenarios. Techniques for coherent feature extraction blending aspects of geometrical integral operators and evolution operators have also received significant attention in the context of non-autonomous dynamical systems \cite{Froyland15,BanischKoltai17,KarraschKeller20}.     

In \cite{Giannakis19,DasGiannakis19}, an interpretation of the timescale separation seen in features recovered from delay-coordinate-mapped data was given through a spectral analysis of kernel integral operators and Koopman evolution operators of dynamical systems \cite{Koopman31,Baladi00,EisnerEtAl15}. Specifically, it was shown that for a measure-preserving ergodic dynamical system, as the number of delays increases, the commutator between kernel integral operators constructed from delay-embedded data (subject to mild requirements) and the Koopman operator converges to zero in operator norm, meaning that these operators acquire common eigenspaces in the infinite-delay limit. Since (i) kernel integral operators associated with sufficiently regular (e.g., continuous) kernels are compact, and thus have finite-dimensional eigenspaces corresponding to nonzero eigenvalues; and (ii) the eigenspaces of Koopman operators of ergodic dynamical systems are one-dimensional, it follows that in the infinite-delay limit, the eigenspaces of the kernel integral operators employed for feature extraction are a finite union of Koopman eigenspaces. The latter are each characterized by a distinct timescale associated with the corresponding eigenvalue of the generator. In applications, it is oftentimes observed that the eigenspaces of kernel integral operators with large numbers of delays are numerically two-dimensional, meaning that they are associated with a single pair of Koopman eigenfrequencies of equal modulus and different sign. Sampled along orbits of the dynamics, such kernel eigenfunctions have the structure of pure sinusoids, which can be thought of as exhibiting an ``ideal'' form of timescale separation.           

A useful aspect of the results in \cite{Giannakis19,DasGiannakis19} is that they hold for broad classes of measure-preserving, ergodic dynamical systems  (including systems with non-smooth attractors) and choices of kernel, and thus provide relevant information about the asymptotic behavior of a variety of feature extraction techniques utilizing delays, including the methods \cite{BroomheadKing86,VautardGhil89,GiannakisMajda11c,GiannakisMajda12a,GiannakisMajda13,BerryEtAl13,ArbabiMezic17,BruntonEtAl17} outlined above. Importantly, the integral operators employed can be consistently approximated in a spectral sense from time series data using well-developed theory  \cite{VonLuxburgEtAl08,TrillosSlepcev18,TrillosEtAl19}.

\subsection{Motivation and contributions of this work}

Despite their generally broad applicability, the results in \cite{Giannakis19,DasGiannakis19} offer limited insight on the behavior of kernel-based feature extraction techniques utilizing delay-coordinate maps for an important class of dynamical systems, namely systems with mixing behavior (or so-called mixed-spectrum systems with both quasiperiodic and mixing components). Indeed, a necessary and sufficient condition for a measure-preserving dynamical system to be mixing is that the generator on the $L^2$ space associated with the invariant measure has a simple eigenvalue at zero, with a constant corresponding eigenfunction, and no other eigenvalues. As a prototypical example, consider the Lorenz 63 (L63) system \cite{Lorenz63} on $\mathbb R^3$, which is rigorously known to possess an ergodic invariant measure $\mu$ supported on the famous ``butterfly'' attractor with mixing dynamics \cite{Tucker99,LuzzattoEtAl05}. According to \cite{DasGiannakis19}, for such a system the kernel integral operator in the infinite-delay limit acquires an infinite-dimensional nullspace containing all $L^2(\mu)$ observables orthogonal to the constant, allowing features with arbitrarily broad frequency spectra (i.e., no timescale separation or coherence). Moreover, data-driven spectral approximation results such as \cite{VonLuxburgEtAl08,TrillosSlepcev18,TrillosEtAl19} do not hold for the potentially infinite-dimensional nullspaces of compact operators.

Yet, as illustrated in Figure~\ref{figPhi}, the eigenfunctions of integral operators based on a sufficiently long delay embedding window, $ T$, exhibit a form of coherence, which can be thought of as a relaxation of the periodic behavior of Koopman eigenfunctions. In particular, for sufficiently large $T$, the time series associated with the kernel eigenfunctions near the top of the spectrum have the structure of amplitude-modulated waves, with a well-defined carrier frequency and a low-frequency modulating envelope. In effect, the pure sinusoids generated by Koopman eigenfunctions can be thought of as special cases of these patterns with constant modulating envelopes. A similar behavior was observed in \cite{SlawinskaGiannakis17}, who found that with increasing number of delays NLSA provides increasingly coherent representations of the El Ni\~no Southern Oscillation of the climate system, as well as other patterns of climate variability. %In Figure~XX, as the number of delays increases, the recovered eigenfunctions eventually lose coherence, as expected from the fact that L63 is a mixing dynamical system witn no nonconstant Koopman eigenfunctions in $L^2(\mu)$.   

\begin{figure}
    \includegraphics[width=\linewidth]{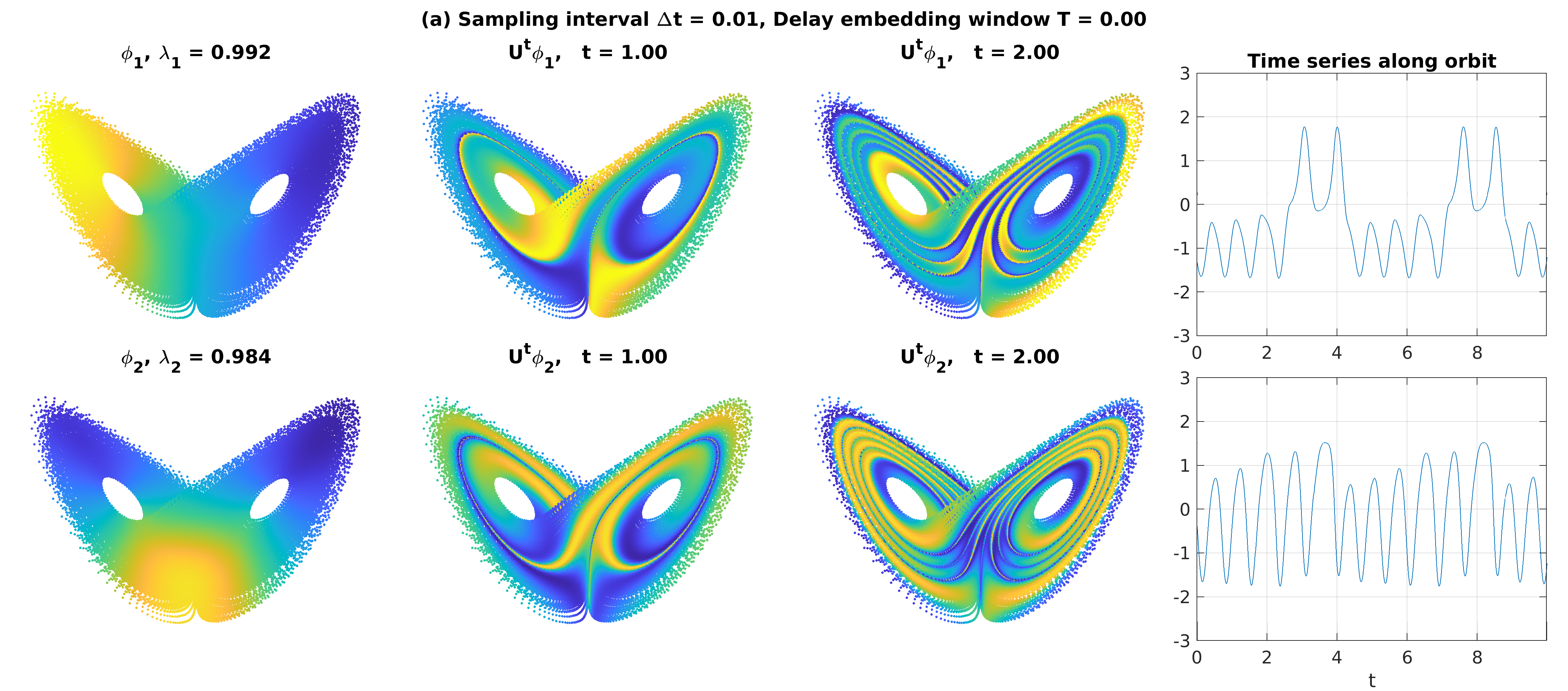}
    \includegraphics[width=\linewidth]{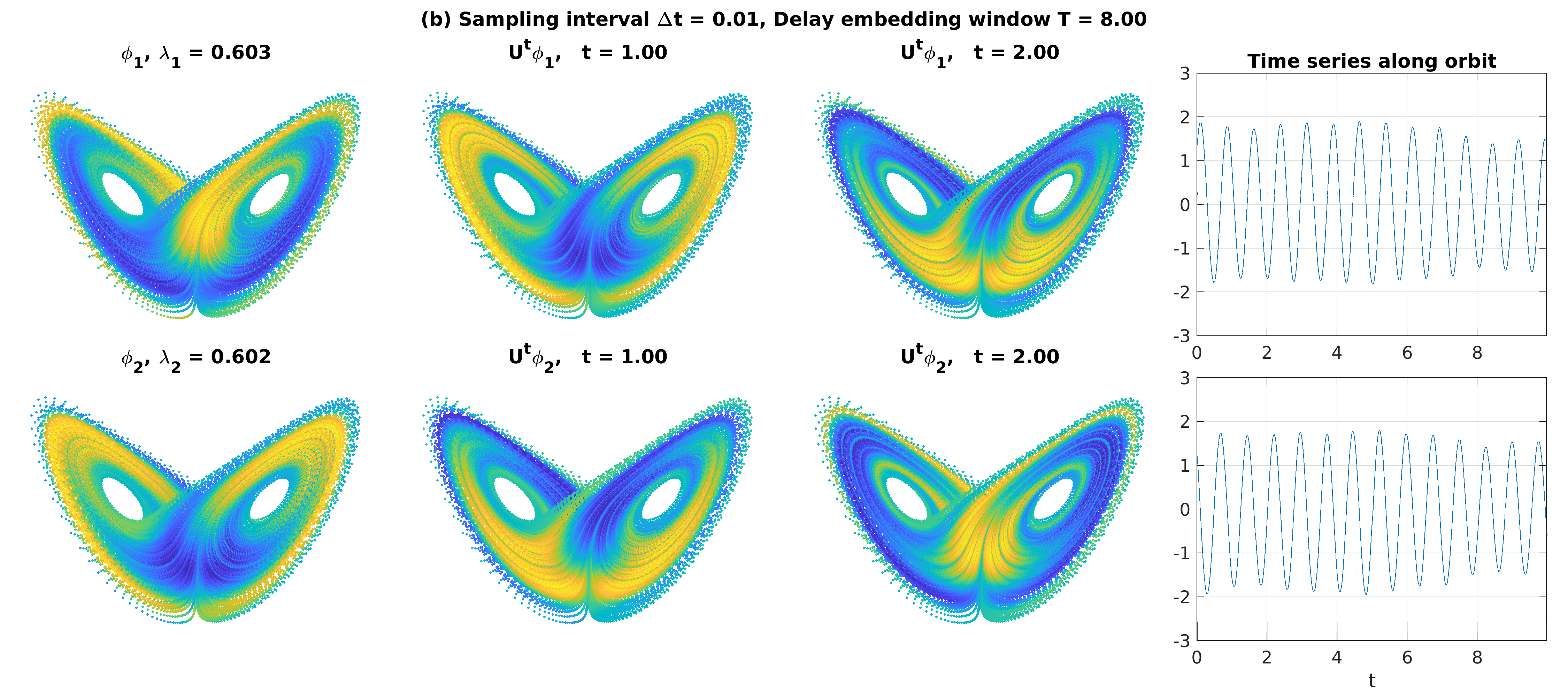}
    \caption{\label{figPhi}Representative eigenfunctions $\phi_{j,T}$ of the integral operator $ K_T $ for (a) no delays, $ T = 0$; and (b) a delay-embedding window $T$ equal to 8 natural time units, numerically approximated from a dataset consisting of $N = \text{64,000}$ samples taken along an orbit of the L63 system at a sampling interval $ \Delta t = 0.01$. In each set of panels, the first and second row show the leading two nonconstant eigenfunctions of $K_T $, in order of decreasing corresponding eigenvalue. The first column from the left shows a scatterplot of $\phi_{j,T}$ on the dataset. The second and third panels show scatterplots of $ \phi_{j,T} $ acted upon by the Koopman operator $U^t$ for time $ t = 1 $ and 2, which corresponds to approximately 1 and 2 Lyapunov characteristic times, respectively. The rightmost column shows a time series of $ \phi_{j,T} $ sampled along a portion of the training trajectory spanning 10 natural time units. The eigenfunctions in (a) exhibit limited dynamical coherence, in the sense that their level sets mix together on times greater than $ \gtrsim 1 $ Lyapunov times. Moreover, their corresponding time series exhibit a broadband frequency spectrum with no apparent phase relationships. In contrast, the eigenfunctions in (b) resist mixing over a period of time spanning multiple Lyapunov times, illustrated by the qualitatively similar nature of the scatterplots of $ \phi_{j,T} $ and  $U^t \phi_{j,T}$ for $ t \in \{ 1, 2 \}$. Furthermore, the time series in (b) have the structure of amplitude-modulated waves with a well-defined carrier frequency and slowly varying modulating envelope, while exhibiting a 90$^\circ$ phase difference to a good approximation.}
\end{figure}

The main contribution of this work is to provide a characterization of the coherence properties of eigenfunctions of integral operators constructed from delay-embedded observables of measure-preserving, ergodic dynamical systems of arbitrary (quasiperiodic, mixing, or mixed-spectrum) spectral characteristics, underpinning the behavior in Figure~\ref{figPhi}. We will do so by studying a class of complex-valued observables $ z $, whose real and imaginary parts are eigenfunctions of an integral operator $K_T : L^2(\mu) \to L^2(\mu)$ constructed using a delay-embedding window of length $T$. These observables will be shown to lie in the $\epsilon$-approximate point spectrum of the Koopman operator $U^t$ for a bound $ \epsilon $ that  decreases at a rate $O(T^{-1})$, but increases with the evolution time $t $ at a linear rate, while also being inversely proportional to the corresponding eigenvalues and the gap between them and the rest of spectrum of $K_T$. Moreover, we give an explicit characterization of the modulating envelope and carrier frequency through the time-autocorrelation function of $ z $ and its derivative at 0, respectively.

For systems possessing non-constant Koopman eigenfunctions, these results imply that at fixed $ t $, $ \epsilon$ can be made arbitrarily small by increasing $T$, so long as $K_T $ satisfies certain positivity conditions that depend on the observation map and the form of the kernel, consistent with the results of \cite{DasGiannakis19}. On the other hand, for systems with mixing dynamics, the behavior of $ \epsilon$, and thus the coherence of $z$, is influenced by an interplay between the delay-embedding window length (promoting coherence) and the decay of the eigenvalues of $ K_T $ with increasing $T$ (inhibiting coherence). Nevertheless, it is possible that $ \epsilon $ is made small by increasing $T$, so long as the eigenvalues associated with $z $ remain sufficiently isolated in the spectrum of $K_T$.     

The plan of this paper is as follows. In Section~\ref{secMain}, we describe the class of dynamical systems under study, and state our results, including Theorem~\ref{thmMain} which is the main theoretical contribution of this work. Section~\ref{secProof} contains a proof of Theorem~\ref{thmMain}, and Section~\ref{secDataDriven} describes the data-driven formulation of our framework. We illustrate our results with numerical examples for the L63 system in Section~\ref{secExamples}, and state our conclusions in Section~\ref{secConclusions}. Auxiliary results and definitions on spectral approximation of integral operators are collected in Appendix~~\ref{appSpecConv}.  

\section{\label{secMain}Main results}

\subsection{\label{secSetup}Dynamical system under study}

Consider a continuous-time, continuous dynamical flow $ \Phi^t : \Omega \to \Omega $, $ t \in \mathbb R $, on a metric space $ \Omega $ possessing an invariant, ergodic Borel probability measure $\mu$, supported on a compact set $ X \subseteq \Omega $. We assume that the support $X$ of the invariant measure is contained in a forward-invariant, $C^1 $ compact manifold $M$ such that $ \Phi^t \rvert_M $ is $C^1$, but do not require that $X$ has differentiable structure. The system is observed through a continuous function $ F : \Omega \to Y$, where $Y$ is a Banach space, and the restriction of $ F $ to $M$ is $C^1$. 

This setup encompasses a large class of autonomous dynamical systems encountered in applications. For instance, as a prototypical ODE example with quasiperiodic behavior, one can consider an ergodic rotation $\Phi^t : \mathbb T^2 \to \mathbb T^2$ on the 2-torus, in which case $ \Omega = M = X = \mathbb T^2$ and $\mu$ is the Haar measure. The L63 system from Figure~\ref{figPhi} is an example of a smooth dissipative flow on $\Omega = \mathbb R^3$, with a rigorously known mixing attractor $X\subset \Omega$ \cite{Tucker99,LuzzattoEtAl05} 
and compact absorbing balls $M \supset X$ \cite{LawEtAl14}. The assumptions stated above also hold for classes of dissipative PDE models possessing inertial manifolds \cite{ConstantinEtAl89}.

Within this class of models, our goal is as follows: Given time-ordered data $ y_0, y_1, \ldots, y_{N-1} \in Y $ with $ y_n = F( x_n )$, sampled along a dynamical trajectory $ x_n = \Phi^{n\,\Delta t}(x_0)$ at an interval $ \Delta t>0$, identify a collection of functions $ \zeta_j : \Omega \to \mathbb C$ which evolve coherently under the dynamics. Intuitively, by that we mean that the dynamically evolved functions $ \zeta_j \circ \Phi^t$ should be relatable to $ \zeta_j $ in a natural way for $ t $ lying in a ``large'' interval containing zero. From the perspective of learning theory, the functions $ \zeta_j $ are principal components/features, which are to be identified through an unsupervised learning problem that favors coherence. Note that this objective differs significantly from the classical proper orthogonal decomposition (POD) \cite{Kosambi43,AubryEtAl91,HolmesEtAl96}, whose goal is to extract features on the basis of explained variance. Once identified, such coherent features are useful in a variety of contexts, including dimension reduction of high-dimensional time series and predictive modeling \cite{ChenEtAl14,AlexanderEtAl17}. In these approaches, a basic premise is that features related to the spectrum of the underlying dynamical system should reveal physically meaningful dynamical processes (e.g., fundamental oscillations of the climate system \cite{SzekelyEtAl16a,SlawinskaGiannakis17}), while having favorable predictability properties. 

\subsection{Pseudospectral criteria for coherence}

To establish a mathematically precise notion of dynamical coherence of observables, consider the evolution group of unitary Koopman operators $U^t : L^2(\mu) \to L^2(\mu) $, acting on observables by composition with the flow, $ U^t f = f \circ \Phi^t $ \cite{Koopman31,KoopmanVonNeumann32,Baladi00,EisnerEtAl15}. By Stone's theorem on one-parameter unitary groups \cite{Stone32}, the group $ \{ U_t \}_{t \in \mathbb R} $ is generated by a skew-adjoint operator $ V : D(V) \to L^2(\mu)$ with a dense domain $D(V) \subset L^2(\mu)$. As an operator, $V$ corresponds to an extension of the directional derivative on $C^1(M) $ functions associated with the vector field $ \vv V $ generating $ \Phi^t $, namely $ \mathcal V f := \vv V \cdot \nabla f $. In particular, for any $ f \in D(V)$, $ t \mapsto U^t f$ is continuously differentiable in $L^2(\mu)$ and
\begin{equation}
    \label{eqGenerator}
    \frac{d\ }{dt} U^t f = V U^t f = U^t V f.
\end{equation}

It is a standard result from ergodic theory \cite{EisnerEtAl15} that whenever $ V $ possesses an eigenfunction $ z \in L^\infty(\mu)$ with $ \lVert z \rVert_{L^2(\mu)} = 1$ and  corresponding eigenvalue $ i \omega$ (where the eigenfrequency $\omega$ is real by skew-adjointness of $V$), then $ \lvert z(x) \rvert = 1 $ for $ \mu $-a.e.\ $x \in \Omega $. Thus, we have the periodic evolution
\begin{equation}
    \label{eqKoopEig}
    U^t z = e^{tV} z = e^{i\omega t} z,
\end{equation}
and at least measure-theoretically, $ U^t z $ can be considered to take values on the unit circle. This means, in particular, that for $ \mu $-a.e.\ $ x \in \Omega$, the time series $ t \mapsto z(\Phi^t(x)) = e^{i\omega t} z(x)$ behaves as a Fourier function on $ \mathbb R $ with frequency $ \omega $. Due to these facts, we think of Koopman eigenfunctions of measure-preserving ergodic dynamical systems as exhibiting an ``ideal'' form of coherence. Indeed, starting from work in the late 1990s on data-driven, spectral analysis of Koopman operators \cite{MezicBanaszuk99,Mezic05} and the related transfer operators \cite{DellnitzJunge99,DellnitzEtAl00} spectral decomposition of evolution operators has emerged as a popular approach for coherent feature extraction in dynamical systems. 

Yet, despite their attractive properties, Koopman eigenfunctions in $L^2(\mu)$ are not an appropriate theoretical paradigm for coherent features of dynamical systems with complex (mixing) behavior. Indeed, a necessary and sufficient condition for a measure-preserving, ergodic flow to be mixing is that the generator $ V $ on $L^2(\mu)$ has a simple eigenvalue 0, with a constant corresponding eigenfunction, and no other eigenvalues. Thus, in this case Koopman eigenfunctions only yield the trivial (constant) feature. 

Systems with so-called mixed spectra exhibit an intermediate behavior, in the sense that they do exhibit non-constant eigenfunctions satisfying~\eqref{eqKoopEig}, but these eigenfunctions span only a strict subspace of $L^2(\mu)$ and provide no information about the mixing component of the dynamics. Specifically, it is a classical result \cite{Halmos56} that $L^2(\mu)$ admits an orthogonal decomposition 
\begin{equation}
    \label{eqL2Decomp}
    L^2(\mu) = H_p \oplus H_c 
\end{equation}
into closed, $U^t$-invariant subspaces $H_p$ and $H_c$, such that every observable in $H_p$ is a linear combination of Koopman eigenfunctions (and thus exhibits a quasiperiodic evolution associated with the point spectrum of the generator), whereas $H_c = H_p^\perp$ is a subspace orthogonal to every Koopman eigenfunction, and thus  associated with the continuous spectrum of the generator. In particular, every observable $g \in H_c$ exhibits a form of mixing behavior (called weak-mixing) characterized by a loss of cross-correlation with any observable $ f \in L^2(\mu)$, viz.,
\begin{equation}
    \label{eqXCorr}
    \lim_{t\to\infty}C_{fg}(t) = 0, \quad \text{where} \quad C_{fg}(t) := \frac{1}{t} \int_0^t \lvert \langle f, U^s g \rangle \rvert \, ds.
\end{equation}
Here, $ \langle \cdot, \cdot \rangle $ denotes the $L^2(\mu)$ inner product, $ \langle f, g \rangle = \int_\Omega f^* g \, d \mu$, taken conjugate-linear in the first argument. The issue with feature extraction by pure Koopman eigenfunctions is that the recovered features cannot capture observables in $H_c$ and their mixing behavior. 

Here, as a natural relaxation of~\eqref{eqKoopEig}, we seek observables satisfying the Koopman eigenvalue equation in an approximate sense. Specifically, we seek nonzero observables $ z \in L^2(\mu)$ satisfying
\begin{equation}
    \label{eqKoopPseudo}
    \lVert U^t z - e^{i\omega t} z \rVert_{L^2(\mu)} \leq \epsilon \lVert z \rVert_{L^2(\mu)},
\end{equation}
for some $\epsilon > 0 $, $ \omega \in \mathbb R$. Every such observable $ z $ is said to be an $ \epsilon $-approximate eigenfunction of $U^t$, and the complex number $e^{i\omega t} $ is said to lie in the $\epsilon$-approximate point spectrum of this operator \cite{Chatelin11}. In addition, we require that the same bound $\epsilon$ holds for all $ t $ in an interval $ [0,\tau ] $ with $ \tau > 0$. Observables satisfying these conditions with $ \epsilon \ll 1 $ and $ \tau \gg 2 \pi/ \omega$ then behave to a good approximation as Koopman eigenfunctions of measure-preserving ergodic dynamical systems. Note, in particular, that the eigenfunctions $\phi_{1,T} $ and $ \phi_{2,T} $ depicted in Figure~\ref{figPhi}(b) are strongly suggestive of this behavior if they are interpreted as the real and imaginary parts of $z$, i.e., $ z = \phi_{1,T} + i \phi_{2,T}$. In the sequel, we will refer to $(e^{i \omega t}, z )$ satisfying~\eqref{eqKoopPseudo} as an $ \epsilon$-approximate eigenpair of $U^t$. It can be shown that because $U^t$ is a normal operator, $(e^{i\omega t}, z ) $ is an eigenpair if and only if it is an $\epsilon$-approximate eigenpair for every $ \epsilon > 0 $.      

\subsection{\label{secKOp}Integral operators induced by delay-coordinate maps}

Motivated by the delay-embedding techniques described in Section~\ref{secIntro}, we seek observables satisfying~\eqref{eqKoopPseudo} through eigenfunctions of integral operators on $L^2(\mu)$ based on delay-coordinate maps. To construct appropriate such operators, consider first the distance-like function $ d : \Omega \times \Omega \to \mathbb R_+$ induced by the norm of $ Y $ and the observable $ F$,  
\begin{displaymath}
    d( x, x' ) = \lVert F(x) - F(x') \rVert_Y, 
\end{displaymath}
and for every $ T > 0 $ define $ d_T : \Omega \times \Omega \to \mathbb R$ with
\begin{equation}
    \label{eqDT}
    d^2_T( x, x' ) =  \frac{1}{ T} \int_0^T d^2(\Phi^t(x),\Phi^t(x')) \, dt.
\end{equation}
The function $d_T $ can be equivalently thought of as being induced from the norm of $ Y_T := L^2([0,T]; Y)$  under the continuous-time delay-coordinate mapping $ F_T: \Omega \to Y_T  $ with $F_T(x)(t) = F(\Phi^t(x))$; that is,  
\begin{displaymath}
    d^2_T(x,x') = \lVert F_T( x ) - F_T( x' ) \rVert_{Y_T}^2 / T.
\end{displaymath}
By convention, we set $d_0 = d$.

Using $d_T $ and a positive, $C^1$, bounded shape function $ h : \mathbb R_+ \to \mathbb R_+ $ with bounded derivative, we then consider the family of symmetric kernel functions $ k_T : \Omega \times \Omega \to \mathbb R_+$, such that
\begin{equation}
    \label{eqKT}
    k_T(x,x') = h(d_T^2(x,x') ).
\end{equation}
As a concrete example, we will nominally work with the choice $ h( u ) = e^{-u/\sigma^2}$, where $\sigma$ is a positive bandwidth parameter. This leads to the radial Gaussian kernel $k_T(x,x') = e^{- d_T^2(x,x')/\sigma^2}$, which is a common starting point in manifold learning techniques \cite{BelkinNiyogi03,CoifmanLafon06} approximating heat kernels on Riemannian manifolds as $\sigma \to 0 $. While here we do not assume that $X$ has manifold structure, which would allow us to use these results, it should be noted when $ Y = \mathbb R^m $ Gaussian kernels have an important property that holds irrespective of the regularity of the support of the sampling distribution of the data, namely they are strictly positive-definite \cite{Steinwart01}. See \cite{Genton01} for additional examples of kernels commonly employed in machine learning applications. 

Every kernel from~\eqref{eqKT} induces an integral operator $K_T : L^2(\mu) \to L^2(\mu)$ such that 
\begin{equation}
    \label{eqKTOp}
    K_T f = \int_{\Omega} k_T( \cdot, x ) f(x ) \, d\mu(x).
\end{equation}
By symmetry and continuity of $k_T$ and compactness of $X$, $K_T $ is a positive-definite, self-adjoint, Hilbert-Schmidt integral operator with Hilbert-Schmidt norm equal to $ \lVert k_T \rVert_{L^2(\mu\times\mu)}$. As a result there exists an orthonormal basis $ \{ \phi_{0,T}, \phi_{1,T}, \ldots \} $ of $L^2(\mu)$ consisting of eigenfunctions of $K_T $ corresponding to the eigenvalues $ \lambda_{0,T} \geq \lambda_{1,T} \geq  \cdots \searrow 0 $. The latter are all real, and have finite multiplicity whenever nonzero by compactness of $K_T$. In addition, by continuous differentiability of $k_T$ and compactness of $X$, every element of in the range of $ K_T $ has a representative in $C^1(M)$. In particular, every eigenfunction $ \phi_{j,T}$ with nonzero corresponding eigenvalue has the continuous representative 
\begin{equation}
    \label{eqVarphi}
    \varphi_{j,T} = \frac{1}{ \lambda_{j,T}} \int_{\Omega} k_T( \cdot, x ) \phi_{j,T}(x) \, d\mu(x),
\end{equation}
whose restriction on $M$ is $C^1$. Note that $ \varphi_{j,T}$ is an everywhere-defined function on $\Omega$, as opposed to the left-hand side of~\eqref{eqKTOp} which is an $L^2(\mu)$-element defined only up to sets of $\mu$-measure zero. We let $ \sigma_p(K_T) = \{ \lambda_{0,T}, \lambda_{1,T},\ldots \} $ denote the point spectrum of $K_T$.  

In the following subsection, we will show that appropriate linear combinations of eigenfunctions $ \phi_{j,T}$ are $\epsilon$-approximate eigenfunctions of the Koopman operator, satisfying~\eqref{eqKoopPseudo} for a threshold $\epsilon$ that decreases as $T $ increases, but increases as $ \lambda_{j,T}$ decreases. The continuous representatives of these eigenfunctions will then provide the coherent features $\zeta_j $.

\begin{remark}
    \label{rkKernel1} In this section, we have opted to work with delay-coordinate maps in continuous time as this will facilitate the derivation of $\epsilon$-approximate spectral bounds valid for continuous time intervals. We will later pass to the more common discrete-time formulation based on the sampling interval $ \Delta t$, which will introduce quadrature errors in~\eqref{eqDT} that vanish as $ \Delta t \to 0$. In addition, aside from the class of radial kernels in~\eqref{eqKT}, our results hold with straightforward modifications to other classes of kernels with $T \to \infty$ limits in $L^2(\mu \times \mu)$. Examples include the covariance kernels employed by SSA (which can  be obtained by polarization of~\eqref{eqKT} using a linear shape function), Markov-normalized kernels \cite{CoifmanLafon06,CoifmanHirn13,BerrySauer16}, and variable-bandwidth kernels \cite{BerryHarlim16}. It is also possible to replace the kernel family $k_T$ in~\eqref{eqKT}, which is obtained by a application of a fixed shape function to the $T$-dependent functions $d_T^2$, by a family $ \tilde k_T $ obtained by averaging a fixed continuous kernel function $ k : \Omega \times \Omega \mapsto \mathbb R$, i.e., $ \tilde k_T(x,x') = \int_0^t k(\Phi^t(x),\Phi^t(x')) \, dt /T$.  See \cite{DasGiannakis19} for further details.   
\end{remark}

\subsection{Dynamically coherent eigenfunctions}

According to the theory of delay-coordinate maps, e.g., \cite{SauerEtAl91,Robinson05,DeyleSugihara11}, for a sufficiently long window, the delay-coordinate map $F_T $ becomes homeomorphic on the compact support $X$ of the invariant measure for a large class of dynamical systems and observation functions $F$, even  if $ F\rvert_X$ is not injective. This property has been widely employed in techniques for state space reconstruction \cite{PackardEtAl80} and forecasting \cite{Sauer93}. Our interest here, however, is not so much on topological reconstruction, but rather on the effect of delay-coordinate maps on the spectral properties of kernel integral operators on $L^2(\mu)$, irrespective of the injectivity properties of $F$. To that end, we begin with a proposition that summarizes some of the results on the limiting behavior of operators in the family $K_T $ from~\eqref{eqKTOp}, reported in \cite{DasGiannakis19}.

\begin{proposition}
    \label{propDelayLimit}As $T \to \infty $, the following hold:
    \begin{enumerate}[(i),wide]
        \item The distance-like functions $d_T$ converge in $L^2(\mu\times\mu)$ norm to a function $ d_{\infty}$, which is invariant under the Koopman operator $U^t \otimes U^t $ of the product dynamical system on $ \Omega \times \Omega$ for any $ t \in \mathbb R$. Correspondingly, the kernel functions $k_T$ also converge in $L^2(\mu\times \mu)$ to a $U^t \otimes U^t$-invariant kernel $k_\infty$. 
        \item The sequence of operators $K_T$ converges in $L^2(\mu)$ operator norm  to the Hilbert-Schmidt integral operator $K_\infty $ associated with $k_\infty$. 
        \item For every $ t \in \mathbb R$, $K_\infty$ and the Koopman operator $U^t $ commute. 
        \item The continuous spectrum subspace $H_c$ lies in the nullspace of $K_\infty$.
    \end{enumerate}
\end{proposition}

While we refer the reader to \cite{DasGiannakis19} for a proof of this proposition, we note here that Claim (i) follows from the fact that with the definition in~\eqref{eqDT}, $d^2_T $ corresponds to a continuous-time Birkhoff average of the continuous function $ d^2 \in C(\Omega \times \Omega)$ under the product dynamical flow $ \Phi^t \times \Phi^t $. The existence and $U^t \otimes U^t$-invariance of $d_\infty$ is then a consequence of the pointwise ergodic theorem. The remaining claims of Proposition~\ref{propDelayLimit} can then be deduced by the $U^t \otimes U^t $-invariance of $k_\infty$. It is also worthwhile noting that, since $ \Phi^t$ is mixing with respect to $ \mu $ if and only if $ \Phi^t \times \Phi^t$ is ergodic with respect to $\mu \times \mu$, it follows that $ d_\infty$ is constant in $L^2(\mu \times \mu )$ sense if and only if the dynamics $ \Phi^t$ is $ \mu $-mixing. In that case, $ d_\infty$ is $\mu\times\mu$-a.e.\ constant by ergodicity, and thus $K_\infty$ is a  kernel integral operator with constant kernel. This implies that the nullspace of $K_\infty$ consists of all $L^2(\mu)$ functions orthogonal to the constant. The latter, comprise precisely the subspace $H_c$ under mixing dynamics, and we conclude that $ \ker K_\infty = H_c$. This last relationship is a special case of Proposition~\ref{propDelayLimit}(iv) for mixing systems.   

For our purposes, the main corollaries of Proposition~\ref{propDelayLimit}, which follow from Claims~(iii) and (ii), respectively, in conjunction with compactness of $K_T $ and $K_\infty$ are:
\begin{corollary}
    \label{corESpace}
    Every eigenspace $E$  of $K_\infty $ corresponding to a nonzero eigenvalue is a finite union of Koopman eigenspaces, and the restriction $V\rvert_E $ of the generator is unitarily diagonalizable. It further follows from skew-adjointness of the generator and ergodicity that $E$ is even-dimensional if and only if is orthogonal to constant functions (i.e., the nullspace of $V$).
 \end{corollary}
 \begin{corollary}
     \label{corSpecConv}
     For every nonzero eigenvalue $ \lambda_{j} $ of $K_\infty$, the sequence of eigenvalues $ \lambda_{j,T} $ of $K_T$ satisfies $ \lim_{T\to\infty}\lambda_{j,T} = \lambda_{j} $. Moreover, the orthogonal projections onto the corresponding eigenspaces converge in operator norm. Conversely, if a sequence $\lambda_T$ of eigenvalues of $K_T$ has a $T \to \infty$ nonzero limit $ \lambda_\infty$, then $ \lambda_\infty$ is necessarily an eigenvalue of $K_\infty$.
\end{corollary}

Suppose now that $E$ is a two-dimensional eigenspace of $K_\infty$ corresponding to a nonzero eigenvalue $ \lambda$, where we have suppressed the $j$ subscript for simplicity of notation. Then, by Corollary~\ref{corESpace}, $E$ is a union of two Koopman eigenspaces orthogonal to $ \ker V$. Let also $ \{ \phi, \psi \} $ be an orthonormal basis of $E $, where the eigenfunctions $ \phi $ and $ \psi $ are real (such a basis can always be found since the kernel $k_\infty$ is real) and $L^2(\mu)$-orthogonal to the constants. Then, it follows by skew-adjointness and reality of $V$ that 
\begin{displaymath}
    \langle \phi, V \phi \rangle = \langle \psi, V \psi \rangle = 0,
\end{displaymath}
whereas 
\begin{displaymath}
    \omega := \langle \psi, V \phi \rangle = - \langle \phi, V \psi \rangle
\end{displaymath}
is real. In addition, $\omega $ is nonzero since $E$ is a $V$-invariant subspace of $L^2(\mu)$ orthogonal to $\ker V$. Defining $z = ( \phi + i \psi ) / \sqrt{2} $, we get
\begin{displaymath}
    V z = \langle \phi, V z \rangle \phi + \langle \psi, V z \rangle = - i \omega \phi + \omega \psi = i \omega z, 
\end{displaymath}
so we conclude that $ z $ is a Koopman eigenfunction corresponding to eigenfrequency $ \omega $. By construction, this eigenfunction has unit $L^2(\mu)$ norm, so for any $ t \in \mathbb R $ we have
\begin{displaymath}
    \alpha_t := \langle z, U^t z \rangle = e^{i\omega t}, 
\end{displaymath}
and if we interpret $ \alpha_t $  as an instantaneous autocorrelation function for $z$  (cf.\ the time-averaged cross-correlation in~\eqref{eqXCorr}), it follows that we can recover Koopman eigenvalues from the time-autocorrelation functions of the corresponding eigenfunctions. It also follows from the generator equation~\eqref{eqGenerator} that $ \omega $ can be determined from the derivative of the autocorrelation function at $ 0 $, $ i \omega = \dot \alpha_t \rvert_{t=0} $. 

Our main result, stated in the form of the following theorem, is essentially a generalization of these basic observations to $ \epsilon$-approximate eigenfunctions of $U^t$ constructed from eigenfunctions of $K_T$ with finite delay-embedding window $T$:

\begin{theorem} 
    \label{thmMain}
    With the assumptions and notation of Sections~\ref{secSetup}--\ref{secKOp}, let $ \phi $ and $ \psi $ be mutually-orthogonal, unit-norm, real eigenfunctions of $K_T$ corresponding to nonzero eigenvalues $ \lambda_T$ and $\nu_T$, respectively, with $ \lambda_T \leq \nu_T$. Assume that $\lambda_T, \nu_T $ are simple if distinct and twofold-degenerate if equal. Define
    \begin{displaymath}
        z = \frac{1}{\sqrt{2}} ( \phi + i \psi ), \quad \alpha_t = \langle z, U^t z \rangle, \quad \omega = \langle \psi, V \phi \rangle\equiv \frac{1}{i} \langle z, V z \rangle \equiv \frac{1}{i} \dot \alpha_t \rvert_{t=0}, 
    \end{displaymath}
    where $\omega $ is real, and set
    \begin{gather*}
        \gamma_T = \min_{ u \in \sigma_p( K_T ) \setminus \{ \lambda_T, \nu_T \} } \left \{  \min \{ \lvert \lambda_T - u \rvert, \lvert \nu_T - u \rvert \} \right \}, \\ 
        \delta_T = \frac{1}{\sqrt 2}( \nu_T - \lambda_T),  \quad \tilde \delta_T = \frac{\delta_T}{\nu_T}.    
    \end{gather*}
    Then, the following hold for every $ t \geq 0 $:
    \begin{enumerate}[(i),wide]
        \item The autocorrelation function $ \alpha_t$ lies in the  $\tilde\epsilon_t$-approximate point spectrum of $U^t$, and $z$ is a corresponding $\tilde \epsilon_t$-approximate eigenfunction for the bound 
            \begin{displaymath}
                \tilde \epsilon_t = s_t + \sqrt{ S_t },  
            \end{displaymath}
            where
            \begin{displaymath}
                s_t = \frac{1}{\gamma_T} \left(  \frac{C_1 t}{T} + 3 \delta_T   \right), \quad S_t =  \frac{ C_2 \lVert \mathcal V \rVert ( 1 + \tilde \delta_T )}{\lambda_T}  \int_0^t s_u \, du. 
            \end{displaymath}
            Here, $\lVert \mathcal V \rVert $ is the norm of the dynamical vector field, viewed as a bounded operator $ \mathcal V : C^1(M) \to C(M)$, and $C_1$ and $C_2$ are constants that depend only on the observation map $F $. Explicitly, we have 
            \begin{displaymath}
                C_1 = 2 \lVert h \rVert_{C^1(\mathbb R_+)} \lVert d^2 \rVert_{C(X\times X)}, \quad C_2 = 2 \lVert h \rVert_{C^1(\mathbb R_+)} \lVert d^2 \rVert_{C^1(M\times M)}.  
            \end{displaymath}
        \item The modulus $ \lvert \omega \rvert$ is independent of the choice of real orthonormal basis $ \{ \phi, \psi \} $ for the eigenspace(s) corresponding to $ \lambda_T $ and $ \nu_T $. Moreover, the phase factor $ e^{i\omega t} $ is related to the autocorrelation function according to the bound 
            \begin{displaymath}
                \lvert \alpha_t - e^{i\omega t } \rvert \leq 2 \sqrt{ S_t }.
            \end{displaymath}
    \end{enumerate}
\end{theorem}

Note that $s_t$ and $S_t$ in Theorem~\ref{thmMain} are increasing functions of $ t \geq 0 $. This, in conjunction with the fact that $ \lVert U^t z - e^{i\omega t} z \rVert_{L^2(\mu)} \leq \lVert U^t z - \alpha_t z \rVert_{L^2(\mu)} + \lvert \alpha_t - e^{i\omega t} \rvert $, leads to the following corollary, which shows how to attain the bound in~\eqref{eqKoopPseudo} valid uniformly over a bounded time interval. 

\begin{corollary}
    \label{corMain} The phase factor $ e^{i\omega t} $ lies in the  $\epsilon_t$-approximate point spectrum of $U^t$, and $z$ is a corresponding $ \epsilon_t$-approximate eigenfunction for the bound 
    \begin{displaymath}
        \epsilon_t = s_t + 3 \sqrt{ S_t }.   
    \end{displaymath}
     Moreover, for every $ \tau \geq 0 $, $ ( e^{i\omega t}, z ) $ is an $ \epsilon_\tau$-approximate eigenpair of $U^t$ for all $ t \in [ 0, \tau ]$. This eigenpair has the continuous representative $ \zeta \in C(\Omega)$ given by
    \begin{displaymath}
        \zeta =  \frac{1}{ \sqrt{2} } \int_\Omega k_T(\cdot, x) \left( \frac{\phi(x)}{\lambda_T} + i \frac{\psi(x)}{\nu_T} \right) \, d\mu(x),
    \end{displaymath}
    which acts as an everywhere-defined, continuous coherent feature on the state space $ \Omega$. 
\end{corollary}

Theorem~\ref{thmMain} will be proved in Section~\ref{secProof}. We now discuss some of the intuitive aspects of the results. First, it should be noted that the bounds established are not sharp, as there are systems for which one can readily construct integral operators $K_T$ with finite embedding windows $T$ and common eigenspaces with the Koopman operator. Examples include operators derived from translation-invariant kernels on tori under quasiperiodic dynamics \cite{Giannakis19,DasGiannakis20}; e.g., the heat kernel associated with the flat metric. For such kernels, there exist eigenfunctions $z$ which are also Koopman eigenfunctions, and the corresponding autocorrelation coefficients $\alpha_t$ lie in the $ \epsilon$-approximate point spectra of $U^t$ for any $ \epsilon > 0 $ and $ t \in \mathbb R $. Still, even without sharp bounds, Theorem~\ref{thmMain} provides useful information on the spectral properties of integral operators utilizing delay-coordinate maps that promote or inhibit dynamical coherence, as follows.
\begin{enumerate}[wide]
    \item As one might expect, the bounds in Theorem~\ref{thmMain} become weaker as the regularity of the observation map $F$ and kernel shape function $ h $ decrease, in the sense that $\tilde \epsilon_t$ and $\epsilon_t$ are increasing functions of the $C^1$ norms of $d^2$ and $h$. It should be noted that many commonly used kernels for feature extraction \cite{Genton01,BelkinNiyogi03,CoifmanLafon06,BerryHarlim16,BerrySauer16}, including the kernels employed in this work, are parameterized by bandwidth parameters controlling the concentration of the kernel about the diagonal (e.g., the parameter $\sigma $ in~\eqref{eqKVB} ahead). For such kernels, the $C^1$ norm of $h$ typically increases without bound as the bandwidth parameter decreases.  
    \item For fixed $t $, the strength of the bounds is an interplay between the length $T$ of the embedding window, the eigenvalue $ \lambda_T$, the gap $ \gamma_T$ (measuring the isolation of the eigenspaces corresponding to $ \lambda_T$ and $ \nu_T$ from the rest of the point spectrum of $K_T$), and the gaps $ \delta_T, \tilde \delta_T$ (measuring the extent at which $\lambda_T$ and $\nu_T$ fail to be twofold-degenerate). Inspecting the dependence of the functions $s_t$ and $S_t$ on these terms indicates that, in general, the bounds become stronger as the window length $T$ increases and/or the gaps $ \delta_T, \tilde \delta_T$ decrease, whereas they weaken as $ \lambda_T$ and/or the gap $ \gamma_T$ decrease. Of course, these terms cannot be independently controlled as $T$ varies, and the expected coherence of $z$ on the basis of Theorem~\ref{thmMain} will depend on their combined effect. It should be noted that Theorem~\ref{thmMain} does not make an assertion about existence of $T \to \infty $ limits for the $\epsilon$-approximate eigenpairs $(e^{i\omega t}, z)$, although as we discuss below there are particular cases for which such limits exist.  
    \item Suppose that the eigenvalue sequence $ \lambda_T$ has a nonzero $ T \to \infty $ limit $\lambda_\infty$. Then, by Proposition~\ref{propDelayLimit}, $ \lambda_\infty$ is a nonzero eigenvalue of the compact operator $K_\infty$. By the same proposition, if the eigenspace $E$ corresponding to $ \lambda_\infty $ does not contain constant functions it is even-dimensional, so the gap coefficients $ \delta_T $ and $\tilde \delta_T$ converge to 0. If, further, $E$ is two-dimensional, the gap $\gamma_T$ converges to a nonzero value. In such cases, Theorem~\ref{thmMain} and Corollary~\ref{corMain} imply that for any  $ \tau \geq 0 $ and $ \epsilon > 0 $, there exists $T_* > 0 $ such that for all $ T > T_* $, \eqref{eqKoopPseudo} holds for all $ t \in [0, \tau]$. This implies in turn that for such a sequence $\lambda_T$ there is a subsequence of frequencies $\omega$ converging to an eigenfrequency of the generator (where we consider a subsequence to account for possible sign flips due the choice of functions $ \phi $ and $\psi$ at each $T$).  Moreover, the corresponding observables $z$ similarly approximate Koopman eigenfunctions.     
    \item Suppose now that the dynamics is mixing with respect to the invariant measure $\mu $. Then, all eigenvalues $ \lambda_T $ with non-constant corresponding eigenfunctions converge to 0 as $ T \to \infty $, and therefore the gaps $ \gamma_T $, $ \delta_T$, and $\tilde \delta_T $ also converge to 0. In that case, the asymptotic behavior of $ \epsilon_t $ as $ T \to \infty$ depends on the behavior of 
        \begin{equation}
            \label{eqEtaT}
            \eta_T :=  \gamma_T \lambda_T T, 
        \end{equation}
    as well as the ratios $ \delta_T / \gamma_T $ and $ \tilde \delta_T \equiv \delta_T / \nu_T$, on the chosen eigenvalue sequences $\lambda_T$ and $\nu_T$. If $\eta_T $ converges to 0 as $T\to \infty$, then $\epsilon_t$ diverges in that limit for any $ t> 0$, failing to provide a useful bound. However, the possibility still remains that the  rate of decay of $ \gamma_T$ and $ \lambda_T $ is slow-enough such that $\eta_T $ attains large values over a suitable range of $T$, allowing $ \epsilon_t$ to remain small on a large interval $ [ 0, \tau ] \ni t $ (so long as $ \gamma_T / \delta_T $ and $\tilde \delta_T $ are also small). In Figure~\ref{figEta}, numerical $\eta_T$ values for the L63 system are found to lie above the value corresponding to the $T=8$ results in Figure~\ref{figPhi} out to at least $\simeq 70 $ Lyapunov times, before eventually decaying. In addition, $\delta_T/\gamma_T$ and $\tilde \delta_T $ are also small after initial transients have died out. Together, these results demonstrate that the bounds from Theorem~\ref{thmMain} are practically relevant for a wide range of delay embedding windows for the L63 system. An intriguing question (lying outside the scope of this work) is whether there are mixing dynamical systems and integral operators for which $\eta_T$ actually diverges as $ T\to \infty $. 
\end{enumerate}

\begin{figure}
    \centering
    \includegraphics[width=\linewidth]{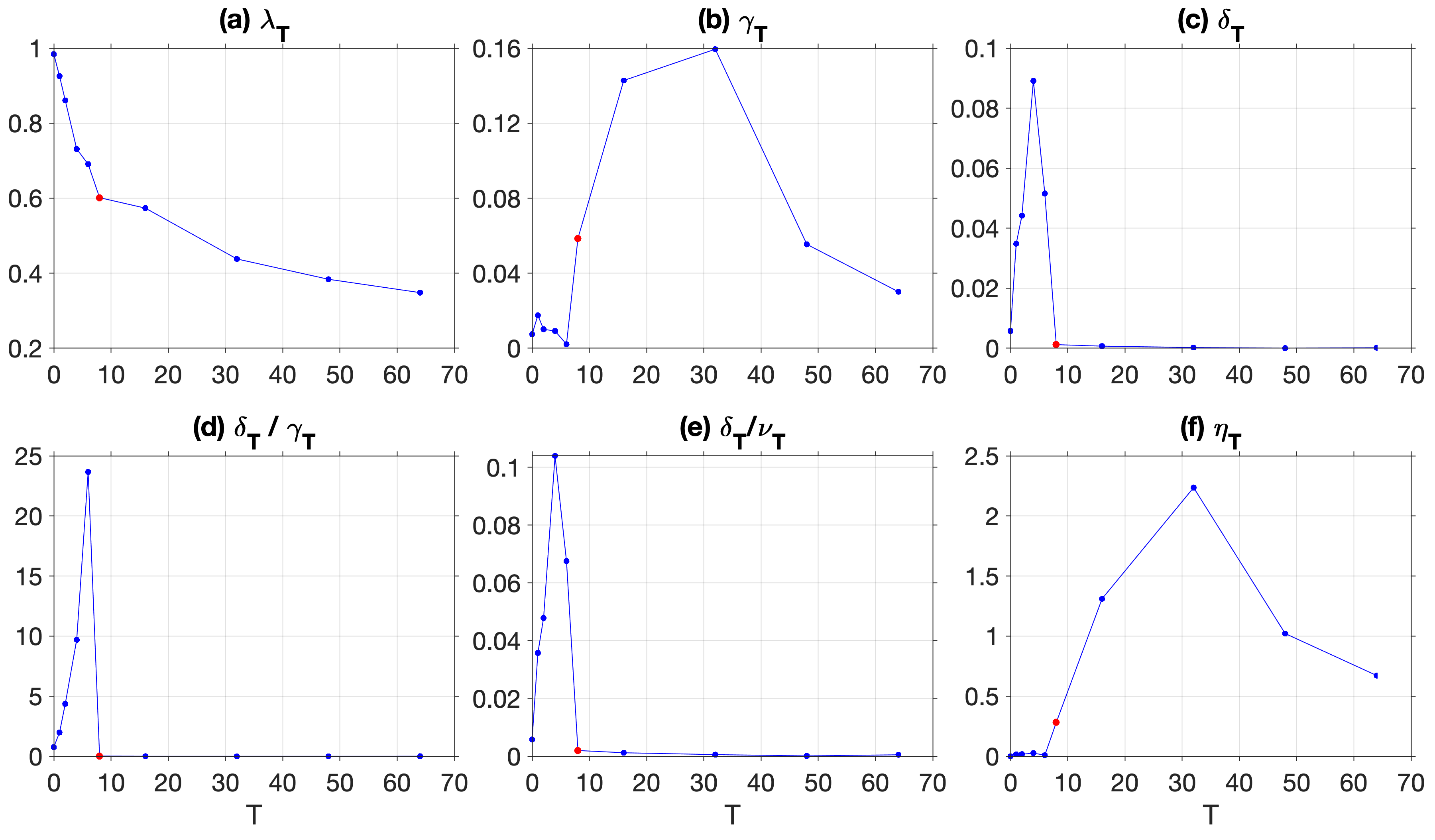}
    \caption{\label{figEta}(a) Eigenvalue~$\lambda_T \equiv \lambda_{2,T}$, (b) spectral gap $\gamma_T$, (c) degeneracy coefficient $\delta_T$, (d, e) ratios $\delta_T/\gamma_T$ and $\tilde \delta_T = \delta_T/ \nu_T $ with $ \nu_T \equiv \lambda_{1,T}$, and  (d) coefficient $\eta_T = \gamma_T \lambda_T T $ from~\eqref{eqEtaT} as a function of the delay-embedding window $T$ for the L63 system. Blue and red markers indicate numerical experiments with $T \in \{ 0, 1, 2, 4, 6, 8, 16, 32, 48, 64 \}$ using datasets of $ N = \text{64,000} $ samples taken at an interval $ \Delta t = 0.01$. Red markers highlight the $T = 8 $ experiment shown in Figure~\ref{figPhi} and Figures~\ref{figLambda}--\ref{figOmega} ahead. The integral operators $K_T$ employ a variable-bandwidth Gaussian kernel with bistochastic (symmetric) Markov normalization, as described in Section~\ref{secDataDriven}.}
\end{figure}

\section{\label{secProof}Proof of Theorem~\ref{thmMain}}

\subsection{\label{secClaim1}Proof of Claim~(i)}
Noting that $z$ and $z^*$ are mutually orthogonal unit vectors in $L^2(\mu)$, and $U^0 = \Id $,  we begin by writing down the expansion
\begin{equation}
    \label{eqUTDecomp}
    U^t z = \alpha_t z + \beta_t z^* + r_t,
\end{equation}
where $ \alpha_t = \langle z, U^t z \rangle$ (as in the statement of the theorem), $ \beta_t = \langle z^*, U^t z \rangle$, $ r_t$ is a residual orthogonal to both $z$ and $z^*$, and
\begin{equation}
    \label{eqABRBound}
    \begin{gathered}
        \lvert \alpha_t \rvert \leq 1, \quad \lvert \beta_t \rvert \leq 1, \quad \lVert r_t \rVert_{L^2(\mu)} \leq 1, \\
        \alpha_0 = 1, \quad \beta_0 = \lVert r_0 \rVert_{L^2(\mu)} = 0.
    \end{gathered}
\end{equation}
It then follows that 
\begin{equation}
    \label{eqAlphaBound}
    \lVert U^tz - \alpha_t z \rVert_{L^2(\mu)} \leq \lvert \beta_t \rvert + \lVert r_t \rVert_{L^2(\mu)},
\end{equation}
and we will prove the first claim of the theorem by bounding $ \lvert \beta_t \rvert$ and $\lVert r_t \rVert_{L^2(\mu)}$.

To that end, note first that by skew-symmetry and reality of $V$, and by definition of the $L^2(\mu)$ inner product,
\begin{displaymath}
    \langle z^*, V z \rangle = - \langle V z^*, z \rangle = - \langle ( V z)^*, z \rangle = - \langle Vz, z^* \rangle^* = - \langle z^*, Vz \rangle, 
\end{displaymath}
so $ \langle z^*, V z \rangle = 0$. Moreover,
\begin{displaymath}
    \langle z, V z \rangle^* = \langle z^*, (V z )^* \rangle = \langle z^*, V z^* \rangle = - \langle V z^*, z^* \rangle = - \langle z^*, V z^* \rangle^* = - \langle z, V z \rangle,
\end{displaymath}
so $ \langle z, V z \rangle$ and $ \langle z^*, V z^* \rangle$ are purely imaginary. In fact, it follows from the definition of $z$ that
\begin{equation}
    \label{eqVZ}
    \langle z, V z \rangle / i = \langle \psi, V \phi \rangle = \omega,
\end{equation}
and from the definition of the generator that
\begin{displaymath}
    \frac{1}{i}\langle z, V z \rangle = \lim_{t \to 0}\frac{1}{it} \langle z, (U^t - \Id) z \rangle = \lim_{t\to 0} \frac{1}{it}(\alpha_t - 1 ) = \dot \alpha_t\rvert_{t=0}, 
\end{displaymath} 
so we can use $ \langle z, V z \rangle / i $ and $ \dot \alpha_t \rvert_{t=0} / i$ as alternative definitions of the frequency $ \omega$ as in the statement of Theorem~\ref{thmMain}.

Using these relationships, the generator equation in~\eqref{eqGenerator}, and the bound for $ \lvert \beta_t \rvert $ in~\eqref{eqABRBound}, we obtain
\begin{align*}
    \frac{d\ }{dt} \lvert \beta_t \rvert^2 &= 2 \Real \left(  \beta_t^* \frac{d \beta_t }{ dt} \right) = 2 \Real \left(  \beta_t^* \frac{d\ }{ dt} \langle z^*, U^t z \rangle \right) =  2 \Real \left( \beta_t^*  \langle z^*, V U^t z \rangle \right) \\
    & = - 2 \Real \left( \beta_t^*  \langle V z^*, U^t z \rangle \right) = - 2 \Real \left( \beta_t^*  \langle V z^*, \alpha_t z + \beta_t z^* + r_t \rangle \right) \\
    & = - 2 \Real \left( \beta_t^* \langle V z^*, r_t \rangle \right) \leq 2 \lvert \beta_t \rvert \lvert \langle V z^*, r_t \rangle \rvert \leq 2  \lVert V z \rVert_{L^2(\mu)} \lVert r_t \rVert_{L^2(\mu)}. 
\end{align*}
Therefore, the squared modulus $\lvert \beta_t \rvert^2$ is bounded by a solution of the differential inequality
\begin{equation}
    \label{eqBetaOde}
    \frac{d\ }{ dt} \lvert \beta_t \rvert^2 \leq \lVert V z \rVert_{L^2(\mu)} \lVert r_t \rVert_{L^2(\mu)}, \quad \beta_0 = 0,  
\end{equation}
where we have used~\eqref{eqABRBound} to set the initial conditions Note that we were able to use the generator equation in order to arrive at this relation since $ z \in \ran K_T $, and every element in $ \ran K_T $ has a $C^1(M)$ representative and thus lies in the domain of the generator, $D(V)$. 

Inspecting~\eqref{eqAlphaBound} and~\eqref{eqBetaOde} indicates that the norm of the residual $ \lVert r_t \rVert_{L^2(\mu)}$ bounds $ \lVert U^t z - \alpha_t z \rVert_{L^2(\mu)}$ both directly, in~\eqref{eqAlphaBound}, and indirectly by bounding the rate of growth of $ \lvert \beta_t \rvert^2$, in~\eqref{eqBetaOde}. In addition, $ \frac{d\ }{dt}\lvert \beta_t \rvert^2$ depends on the norm $ \lVert V z \rVert_{L^2(\mu)}$. The following two lemmas are useful for estimating these terms. 

\begin{lemma}
    \label{lemmaCommutator}
    With the notation and assumptions of Theorem~\ref{thmMain}, for every $ t \geq 0 $ and $ T > 0 $ the commutator $[U^t, K_T ]$ satisfies 
    \begin{displaymath}
        \lVert [U^t, K_T] \rVert \leq \frac{2 \lVert h \rVert_{C^1(\mathbb R_+)} \lVert d \rVert^2_{C(X\times X)} t }{T},
    \end{displaymath}
    where $ \lVert \cdot \rVert $ denotes $L^2(\mu)$ operator norm. 
\end{lemma}

\begin{proof} The proof follows closely that of Lemma~19 in \cite{DasGiannakis19}, which established a similar result for discrete-time sampling and $C(X)$ operator norm. In particular, it is a direct consequence of the definition of the delay-coordinate distance $d_T$ in~\eqref{eqDT} that for any $ x, x' \in X $ and  $ t \geq 0 $,
    \begin{align*}
        d^2_T(\Phi^t(x),\Phi^t(x')) & = \frac{1}{T} \int_0^T d^2(\Phi^{t+u}(x),\Phi^{t+u}(x')) \, du \\
        &= d^2_T(x,x') + \frac{1}{T} \left( \int_T^{T+t} du - \int_0^t du \right) d^2 ( \Phi^u(x),\Phi^u(x') ).
    \end{align*}
    Therefore, 
    \begin{displaymath}
        \lvert d^2_T(\Phi^t(x),\Phi^t(x')) - d^2_T(x,x') \rvert \leq \frac{2 \lVert d \rVert^2_{C(X\times X)}t}{T},
    \end{displaymath}
    and using the above and the definition of the kernel $k_T$ in~\eqref{eqKT} we get
    \begin{align}
        \nonumber \lvert k_T(\Phi^t(x),\Phi^t(x')) - k_T(x,x') \rvert &= \lvert h( k_T(\Phi^t(x),\Phi^t(x')) ) - h(k_T(x,x')) \rvert \\
        \nonumber & \leq \lVert h \rVert_{C^1(\mathbb R_+)} \lvert d^2_T(\Phi^t(x),\Phi^t(x')) - d^2_T(x,x') \rvert \\
        \label{eqKTBound} &\leq \frac{2 \lVert h \rVert_{C^1(\mathbb R_+)}\lVert d \rVert^2_{C(X\times X)}t}{T}.
    \end{align}
    It then follows that for any $ f \in L^2(\mu) $
    \begin{align*}
        \lVert U^t K_T f - K_T U^t f \rVert_{L^2(\mu)} &= \left \lVert  \int_\Omega \left(  k_T(\Phi^t(\cdot),x) f(x) - k_T(\cdot, x) f(\Phi^t(x)) \right) \, d\mu(x) \right \rVert_{L^2(\mu)} \\
        &= \left \lVert  \int_\Omega \left(  k_T(\Phi^t(\cdot),\Phi^t(x)) - k_T(\cdot, x)  \right) U^t f(x) \, d\mu(x) \right \rVert_{L^2(\mu)} \\
        &\leq \lVert k_T(\Phi^t(\cdot),\Phi^t(\cdot)) - k_T \rVert_{C(X\times X)}\lVert U^t f \rVert_{L^1(\mu)} \\
        &\leq \lVert k_T(\Phi^t(\cdot),\Phi^t(\cdot)) - k_T \rVert_{C(X\times X)}\lVert f \rVert_{L^2(\mu)}. 
    \end{align*}
    Note that to obtain the second and last lines in the displayed equations above we used the fact that $ \mu $ is an invariant probability measure under the flow $ \Phi^t$. Using this result and~\eqref{eqKTBound}, we arrive at
    \begin{displaymath}
        \lVert [U^t, K_T] \rVert = \lVert U^t K_T - K_T U^t \rVert \leq \frac{2 \lVert h \rVert_{C^1(\mathbb R_+)}\lVert d \rVert^2_{C(X\times X)}t}{T}, 
    \end{displaymath}
    proving the lemma. \qed
\end{proof}
 
\begin{lemma}
    \label{lemmaVBound}
    With the notation and assumptions of Theorem~\ref{thmMain}, the family of operators $ \{ A_T = V K_T \mid T > 0 \} $ is uniformly bounded on $L^2(\mu)$ with
    \begin{displaymath}
        \lVert A_T \rVert \leq \lVert \mathcal V \rVert \lVert h \rVert_{C^1(\mathbb R_+)} \lVert d^2 \rVert_{C^1(M\times M)}.
    \end{displaymath}
\end{lemma}

\begin{proof}
    We use the notation $ \mathcal V_1 : C^1(M \times M) \to C(M \times M) $ to represent the differential operator on $C^1(M\times M)$ which acts by the dynamical vector field $ \mathcal V : C^1(M) \to C(M) $ along the first coordinate; i.e., 
    \begin{displaymath}
        \mathcal V_1 f(x,x') = \lim_{t\to 0}\frac{f(\Phi^t(x),x')}{t} = \mathcal V f_{x'}(x), 
    \end{displaymath}
    where $ f_{x'} = f( \cdot, x' ) \in C^1(M)$. Note that $ \mathcal V_1$ and $ \mathcal V$ have equal operator norms, $ \lVert \mathcal V_1 \rVert = \lVert \mathcal V \rVert$. Moreover, $ \mathcal V_1 $ commutes with the induced action by the product dynamical flow $ \Phi^t \otimes \Phi^t $ on $C^1(M\times M)$, in the sense that
    \begin{displaymath}
        \mathcal V_1 ( f \circ \Phi^t \otimes \Phi^t ) = ( \mathcal V_1 f ) \circ \Phi^t \otimes \Phi^t, \quad \forall t \geq 0, \quad \forall  f \in C^1(M\times M).
    \end{displaymath}
    Using these facts, we obtain
    \begin{align*}
        \lVert \mathcal V_1 d^2_T \rVert_{C(X \times X)} & = \left \lVert \frac{1}{T} \int_0^T \mathcal V_1( d^2 \circ \Phi^t \otimes \Phi^t)  \, dt \right \rVert_{C(X \times X)} \\
        &= \left \lVert\frac{1}{T}  \int_0^T ( \mathcal V_1 d^2 )  \circ \Phi^t \otimes \Phi^t \, dt \right \rVert_{C(X \times X)} \\
        & \leq \lVert \mathcal V_1 d^2 \rVert_{C(X \times X)} \leq \lVert \mathcal V_1 \rVert \lVert d^2 \rVert_{C^1(M\times M)} = \lVert \mathcal V \rVert \lVert d^2 \rVert_{C^1(M\times M)},
    \end{align*}
    and thus
    \begin{align}
        \nonumber \lVert \mathcal V_1 k_T \rVert_{C(X\times X)} & = \lVert \mathcal V_1 ( h \circ d_T^2 ) \rVert_{C(X\times X)} \leq \lVert h \rVert_{C^1(\mathbb R_+)} \lVert \mathcal V_1 d_T^2 \rVert_{C(X \times X)} \\
        \label{eqV1KTBound} &\leq \lVert h \rVert_{C^1(\mathbb R_+)} \lVert \mathcal V \rVert \lVert d^2 \rVert_{C^1(M\times M)}.   
    \end{align}

    Now, because $k_T $ lies in $C^1(M\times M)$, for every $ f \in L^2(\mu)$ we have
    \begin{displaymath}
        A_T f = V K_T f = V \int_\Omega k_T( \cdot, x ) f(x) \, d\mu(x) = \int_\Omega \mathcal V_1 k_T(\cdot, x) f(x) \, d\mu(x),
    \end{displaymath}
    so $A_T$ is a kernel integral operator on $L^2(\mu)$ whose kernel $ \mathcal V_1 k_T $ is continuous on $X \times X$. The $ L^2(\mu) $ operator norm of $A_T$ therefore satisfies
    \begin{displaymath}
        \lVert A_T \rVert \leq \lVert \mathcal V_1 k_T \rVert_{C(X \times X)},
    \end{displaymath}
    and the claim of the lemma follows from~\eqref{eqV1KTBound}. \qed
\end{proof}

With these results in place, we proceed to bound $ \lVert r_t \rVert_{L^2(\mu)}$. First, acting with $K_T$ on both sides of~\eqref{eqUTDecomp}, we obtain
\begin{align*}
    K_T U^t z &= \alpha_t K_T z + \beta_t K_T z^* + K_T r_t \\
    &=  \frac{ \alpha_t }{ \sqrt{2} } ( \lambda_T \phi + i \nu_T \psi ) + \frac{ \beta_t }{ \sqrt{2} } ( \lambda_T \phi - i \nu_T \psi) + K_T r_t \\
    &= \lambda_T ( \alpha_t z + \beta_t z^* ) + i \delta_T (\alpha_t - \beta_t ) \psi + K_T r_t \\
    &= \lambda_T U^t z +  i \delta_T (\alpha_t - \beta_t ) \psi + (K_T - \lambda_T ) r_t \\
    &= \frac{1}{\sqrt{2}} U^t(K_T \phi + i \lambda_T \psi ) +  i \delta_T (\alpha_t - \beta_t ) \psi + (K_T - \lambda_T ) r_t \\
    &= \frac{1}{\sqrt{2}} U^t(K_T \phi + i K_T \psi ) +  i \delta_T (\alpha_t - \beta_t - U^t ) \psi + (K_T - \lambda_T ) r_t \\
    &= U^t K_T z +  i \delta_T(\alpha_t - \beta_t - U^t ) \psi + (K_T - \lambda_T ) r_t. 
\end{align*}
Therefore,
\begin{displaymath}
    (K_T - \lambda_T) r_t = - [ U^t,K_T] z + i \delta_T (U^t - \alpha_t - \beta_t) \psi.
\end{displaymath}
which, in conjunction with~\eqref{eqABRBound}, leads to 
\begin{equation}
    \label{eqRTLowerBound}
    \lVert (K_T - \lambda_T) z \rVert_{L^2(\mu)} \leq \lVert [ U^t, K_T] \rVert + 3 \delta_T. 
\end{equation}
On the other hand,
\begin{align}
    \nonumber \lVert ( K_T - \lambda_T ) r_t \rVert_{L^2(\mu)} &= \sum_{\lambda_{j,T} \in \sigma_p(K_T) \setminus \{ \lambda_T, \nu_T \}} ( \lambda_{j,T} - \lambda_T )^2 \lvert \langle \phi_{j,T}, r_t \rangle \rvert^2 \\
    \label{eqRTUpperBound}& \geq \sum_{\lambda_{j,T} \in \sigma_p(K_T) \setminus \{ \lambda_T, \nu_T \}} \gamma_T^2 \lvert \langle \phi_{j,T}, r_t \rangle \rvert^2 = \gamma_T^2 \lVert r_t \rVert_{L^2(\mu)}^2,
\end{align}
and using~\eqref{eqRTLowerBound}, \eqref{eqRTUpperBound}, and Lemma~\ref{lemmaCommutator}, we arrive at the bound 
\begin{equation}
    \label{eqRBound}
    \lVert r_t \rVert_{L^2(\mu)} \leq \frac{1}{\gamma_T} \left( \frac{2 \lVert h \rVert_{C^1(\mathbb R_+)} \lVert d \rVert^2_{C(X\times X)} t }{T} + 3 \delta_T \right) = s_t,
\end{equation}
where the function $s_t $ was defined in the statement of Theorem~\ref{thmMain}.

Next, it follows from Lemma~\ref{lemmaVBound} that 
\begin{align}
    \nonumber\lVert V z \rVert_{L^2(\mu)}  &=  \frac{1}{\sqrt{2}} \lVert  V(\phi+i\psi )  \rVert_{L^2(\mu)} =  \frac{1}{\sqrt{2}} \left \lVert  V K_T \left( \frac{ \phi }{ \lambda_T }+ i \frac{ \psi }{ \nu_T } \right) \right \rVert_{L^2(\mu)} \\
    \nonumber &= \frac{1}{\lambda_T} \left \lVert A_T \left ( z  + \frac{i}{ \sqrt{2}} \left( \frac{1}{\nu_T} - \frac{1}{\lambda_T} \right)\psi \right ) \right  \lVert_{L^2(\mu)} \\
    \label{eqVZBound} &\leq \frac{ \lVert A_T \rVert (1 + \tilde \delta_T) }{ \lambda_T } = \frac{ \lVert \mathcal V \rVert \lVert h \rVert_{C^1(\mathbb R_+)} \lVert d^2 \rVert_{C^1(M\times M)} ( 1 + \tilde \delta_T)}{\lambda_T}. 
\end{align}
Inserting the estimates for $ \lVert r_t \Vert_{L^2(\mu)}$ and $ \lVert Vz \rVert_{L^2(\mu)}$ in~\eqref{eqRBound} and~\eqref{eqVZBound}, respectively, into~\eqref{eqBetaOde}, and using the definition of the constant $C_2$ in the statement of the theorem, then leads to the differential inequality
\begin{displaymath}
    \frac{d\ }{dt} \lvert \beta_t \rvert^2 \leq \frac{C_2 \lVert \mathcal V \rVert (1 + \tilde \delta_T)}{\lambda_T} s_t, \quad \beta_0 = 0,
\end{displaymath}
and integrating we obtain
\begin{equation}
    \label{eqBetaBound}
    \lvert \beta_t \rvert^2 \leq \frac{C_2 \lVert \mathcal V \rVert (1 + \tilde \delta_T)}{\lambda_T} \int_0^t s_u \, du = \frac{ 2 C_2 (1 + \tilde \delta_T)}{\lambda_T \gamma_T}  \left( \frac{ C_1 t^2 }{T} + 3 \delta_T t \right) = S_t,
\end{equation}
where the function $S_t $ is defined in the statement of the theorem. Substituting~\eqref{eqRBound} and~\eqref{eqBetaBound} into~\eqref{eqAlphaBound} then leads to $\lVert U^t z - \alpha_t z \rVert_{L^2(\mu)} \leq s_t + \sqrt{ S_t } $, proving Claim~(i) of the theorem. 

\subsection{Proof of Claim~(ii)}

First, to verify that $ \omega $ is independent of the choice of mutually orthonormal basis functions $ \phi $ and $ \psi $, it is sufficient to consider the following two cases:
\begin{itemize}
    \item Case I: $ \lambda_T$ and $ \nu_T $ are simple eigenvalues. In this case, the claim is obvious since any unit-norm eigenvectors $ \phi' $ and $ \psi' $  corresponding to $\lambda_T$ and $\nu_T$, respectively, are related to $ \phi $ and $ \psi $ by 
        \begin{displaymath}
            \phi' = c_\phi \phi, \quad \psi' = c_\psi \psi, 
        \end{displaymath}
        where $ c_\phi, c_\psi \in \{ -1, 1 \} $. 
    \item Case II: $ \lambda_T = \nu_T $ are twofold-degenerate eigenvalues. To verify the claim, let $ \{ \phi', \psi' \}$ be any real, orthonormal basis of the corresponding eigenspace, $E$. Then, there exists a $ 2 \times 2 $ orthogonal matrix $ O $ such that 
        \begin{displaymath}
            \begin{pmatrix} 
                \phi' \\ 
                \psi' 
            \end{pmatrix}
            = O 
            \begin{pmatrix}
                \phi \\
                \psi
            \end{pmatrix},
            \quad
            O = \begin{pmatrix} 
                O_{\phi\phi} & O_{\phi\psi}  \\
                O_{\psi\phi} & O_{\psi\psi}
            \end{pmatrix}.
        \end{displaymath}
        Since  $\langle \phi, V \phi \rangle = \langle \psi, V  \psi \rangle = 0 $ (by skew-adjointness and reality of $V$, in conjunction with reality of $ \phi $ and $ \psi $), we have
        \begin{align*}
            \lvert \langle \psi', V \phi' \rangle \rvert  &= \lvert \langle O_{\psi \phi} \phi + O_{\psi\psi}\psi, O_{\phi\phi} V\phi + O_{\phi\psi} V \psi \rangle \rvert = \lvert (  O_{\psi\phi} O_{\phi\psi} - O_{\phi\phi} O_{\psi\psi} ) \omega \rvert \\
            &= \lvert \det O \rvert \lvert \omega \rvert = \lvert \omega \rvert, 
        \end{align*}
        proving that $ \lvert \omega \rvert $ is independent of the choice of real orthonormal basis of $E$. 
\end{itemize}

Next, to bound $ \lvert \alpha_t - e^{i\omega t} \rvert$, we follow a differential inequality approach similar to that used to bound $ \lvert \beta_t \rvert$ in Section~\ref{secClaim1}. In particular, let $ a_t = \alpha_t - e^{i\omega t}$. We have
\begin{displaymath}
    \lvert a_t \rvert^2  = \lvert \alpha_t \rvert^2 + 1 - 2 \Real(\alpha_t e^{-i\omega t}),
\end{displaymath}
and therefore 
\begin{equation}
    \label{eqAAT}
    \frac{d\ }{dt} \lvert a_t \rvert^2 \leq \frac{d\ }{dt} \lvert \alpha_t \rvert^2 + 2 \left \lvert \Real \frac{d\ }{dt}  \left( \alpha_t e^{-i \omega t} \right) \right \rvert.
\end{equation}
To place a bound on the first term in the right-hand side of~\eqref{eqAAT}, observe that 
\begin{align}
    \nonumber \frac{d\ }{dt} \lvert \alpha_t \rvert^2 &= 2 \Real \left( \alpha_t^* \frac{d\alpha_t}{dt} \right) = 2 \Real\left( \alpha_t^* \langle z, U^t V z \rangle  \right) = - 2 \Real\left(  \alpha_t^* \langle V z, U^t z \rangle \right) \\
    \nonumber &= - 2 \Real\left( \alpha_t^* \langle Vz, \alpha_t z + \beta_t z^* + r_t \rangle \right) \\
    \nonumber &= - 2 \Real\left( \lvert \alpha_t \rvert^2 \langle V z, z \rangle + \alpha_t^* \beta_t \langle V z, z^* \rangle + \alpha_t^* \langle V z, r_t \rangle \right) \\
    \label{eqAAT1} &= - 2 \Real\left( \alpha_t^* \langle V z, r_t \rangle \right) \leq 2 \lvert \alpha_t \rvert \lvert \langle V z, r_t \rangle \rvert \leq 2 \lvert \langle V z, r_t \rangle \rvert.
\end{align}
Note that to obtain the equality in the second-to-last line we used the facts that $ \langle z^*, V z \rangle $ and $ \langle z, V z \rangle$ are vanishing and purely imaginary, respectively (see Section~\ref{secClaim1}). Moreover, we used the bound $ \lvert \alpha_t \rvert \leq 1$ in~\eqref{eqABRBound} and Lemma~\ref{lemmaVBound} to arrive at the inequality in the last line. Similarly, using~\eqref{eqVZ} and the fact that  $ \langle z^*, V z \rangle = 0 $, leads to
\begin{align}
    \nonumber \left \lvert \Real \frac{d\ }{dt} \left( \alpha_t e^{-i \omega t} \right) \right \rvert &= \left \lvert \Real \left(  \langle V z, \alpha_t z + \beta_t z^* + r_t \rangle e^{-i \omega t } - i \omega \alpha_t e^{- i\omega t} \right) \right \rvert \\
    \label{eqAAT2}&= \left \lvert \Real( \langle V z, r_t \rangle e^{-i \omega t} ) \right \rvert \leq \lvert \langle Vz, r_t \rangle \rvert,
\end{align}
and inserting~\eqref{eqAAT1} and~\eqref{eqAAT2} into~\eqref{eqAAT}, we obtain
\begin{displaymath}
    \frac{d\ }{dt}\lvert a_t \rvert^2 \leq 4 \lvert \langle Vz, r_t \rangle \rvert = \frac{4 C_2 \lVert \mathcal V \rVert (1 + \tilde \delta_T)}{\lambda_T} s_t.
\end{displaymath}
Integrating this differential inequality subject to the initial condition $ a_0 = \alpha_0 - 1 = 0 $ then leads to
\begin{displaymath}
    \lvert \alpha_t - e^{i\omega t} \rvert^2 = \lvert a_t \rvert^2 \leq 4 S_t,
\end{displaymath}
and the bound in Claim~(ii) of Theorem~\ref{thmMain} follows.  This completes our proof of the theorem.

\section{\label{secDataDriven}Data-driven approximation}

In this section, we consider how to approximate the eigenvalues and eigenfunctions of the integral operator $K_T $, as well as the frequency $\omega$ and autocorrelation function $ \alpha_t$, from the time series data $ y_0, \ldots, y_{N-1}$ sampled at the interval $\Delta t$, as described in Section~\ref{secSetup}. Aside from errors associated by approximating continuous-time delay-coordinate maps by (their more familiar) discrete-time analogs, error analyses for the approximation scheme described below have been performed elsewhere \cite{DasGiannakis19,GiannakisEtAl19,DasEtAl20}. Here, we limit ourselves to a high-level description of the construction and its convergence in the large-data limit, relegating technical details to these references.  

\subsection{\label{secDataDrivenDesc}Construction of the data-driven approximation scheme}

%For the purposes of data-driven approximation, we will strengthen our regularity assumptions over those stated in Section~\ref{secMain}, requiring that the manifold $M$, flow $ \Phi^t|_M$, observation map $F|_M$, and kernel shape function $h$ all have $C^{1,\alpha}$ H\"older regularity for some $ \alpha \in ( 0, 1 ]$. With this assumption, 

The main steps in the construction of the approximation scheme are as follows:

\paragraph{Step~1 (Discrete-time delay-coordinate map)} Replace the continuous-time delay-coordinate map $F_T : \Omega \to L^2([0,T]; Y)$ by the discrete-time map $F_{Q,\Delta t} : \Omega \to Y^Q$ given by
\begin{equation}
    \label{eqFQ}
    F_{Q,\Delta t}(x) = ( F(x), F(\Phi^{\Delta t}(x)), \ldots, F(\Phi^{(Q-1)\,\Delta t}(x))).
\end{equation}
Here, $Q$ is an integer parameter corresponding to the number of delays. The map $F_{Q,\Delta t}$ with $ \Delta t = T / Q$ then induces a continuous distance-like function $ d_{T,\Delta t} : \Omega \times \Omega \to \mathbb R_+$, 
\begin{displaymath}
    d^2_{T,\Delta t}(x,x') = \frac{1}{Q} \lVert F_{Q,\Delta t}(x) - F_{Q,\Delta t}(x') \rVert_{Y^Q}^2 = \frac{1}{Q} \sum_{q=0}^{Q-1} d^2(\Phi^{(q-1)\,\Delta t}(x), \Phi^{(q-1)\,\Delta t}(x')),   
\end{displaymath}
which is meant to approximate continuous-time function $d_T$ from~\eqref{eqDT}. Specifically, standard properties of quadrature using the rectangle rule \cite{DavisRabinowitz84} lead to the estimates
\begin{align}
    \label{eqDTConvergence}
    \lVert d^2_T - d^2_{T, \Delta t} \rVert_{C(M\times M)} & \leq \frac{\lVert d^2 \rVert_{C^1(M\times M)} \, \Delta t}{2} = \frac{\lVert d^2 \rVert_{C^1(M\times M)} T}{2 Q}, \\
    \label{eqDT1Convergence}
    \lVert d^2_T - d^2_{T,\Delta t} \rVert_{C^1(M\times M)} &= o( \Delta t^0). 
\end{align}
Similarly, we approximate the continuous-time kernel $k_T$ in~\eqref{eqKT} by $ k_{T,\Delta t} := h \circ d^2_{T,\Delta t}$. Note that \eqref{eqDT1Convergence} merely indicates that as $ \Delta t \to 0$, $ d^2_{T,\Delta t}$ converges to $ d^2_T$ in $C^1$ norm. A stronger bound can be obtained if $d_T$ has higher than $C^1 $ regularity, e.g., $ \lVert d^2_T - d^2_{T,\Delta t} \rVert_{C^1(M\times M)} = O( \Delta t^\alpha ) $ if it lies in $C^{1,\alpha}(M \times M) $ for some $ \alpha > 0 $.   

\paragraph{Step~2 (Sampling measure)} Replace the Hilbert space $L^2(\mu)$ associated with the invariant measure with the finite-dimensional Hilbert space $L^2(\mu_N)$ associated with the sampling measure $ \mu_N = \sum_{n=0}^{N-1} \delta_{x_n}$ on the dynamical trajectory $ x_0, \ldots, x_{N-1} \in \Omega$ underlying the data $ y_0, \ldots, y_{N-1} $. Here, $ \delta_x$ denotes the Dirac measure supported at $ x \in \Omega$. The space $L^2(\mu_N)$ consists of equivalence classes of measurable, complex-valued functions on $ \Omega$ with common values at the sampled states $ x_0, \ldots, x_N$, and is equipped with the inner product 
\begin{displaymath}
    \langle f, g \rangle_{N} = \int_\Omega f^* g \, d\mu_N = \frac{1}{N} \sum_{n=0}^{N-1} f^*(x_n) g(x_n). 
\end{displaymath}
For simplicity of exposition, we will assume that all sampled states $x_n $ are distinct (by ergodicity, this will be the case aside from trivial cases), so $L^2(\mu_N)$ is an $N$-dimensional Hilbert space, canonically isomorphic to $ \mathbb C^N$ equipped with a normalized dot product. Under this isomorphism, an element $ f \in L^2(\mu_N)$ is represented by a column vector $ \vv f = (f_0, \ldots, f_{N-1})^\top \in \mathbb C^N$ such that $ f_n = f(x_n)$, and we have $ \langle f, g \rangle_{N} = \vv f \cdot \vv g / N$. Moreover, a linear map $A : L^2(\mu_N) \to L^2(\mu_N)$ is represented by an $N\times N$ matrix $ \bm A $ such that $ \bm A \vv f$ corresponds to the column vector representation of $ A f$. We will also assume without loss of generality that the starting state $x_0$ (and thus the entire sampled dynamical trajectory) lies in the forward-invariant manifold $M$, but note that $x_0 $ need not lie on the support $X$ of the invariant measure.  In light of these facts, our data-driven schemes can be numerically implemented using standard tools from linear algebra, and as we will see below, their formulation requires few structural modifications of their infinite-dimensional counterparts from Section~\ref{secMain}.    

\paragraph{Step~3 (Data-driven integral operator)} Approximate the kernel integral operator $K_T : L^2(\mu) \to L^2(\mu)$ by the operator $ K_{T,\Delta t,N} : L^2(\mu_N) \to L^2(\mu_N)$, where
\begin{displaymath}
    K_{T,\Delta t,N} f = \int_{\Omega} k_{T,\Delta t}(\cdot, x)f(x) \, d\mu_N(x) = \frac{1}{N} \sum_{n=0}^{N-1} k_{T,\Delta t}(\cdot,x_n)f(x_n).
\end{displaymath}
This operator is self-adjoint, and there exists a real orthonormal basis $ \{ \phi_{0,T,\Delta t,N}, \ldots, \phi_{N-1,T,\Delta t,N} \} $ of $L^2(\mu_N)$ consisting of its eigenvectors, with corresponding eigenvalues $ \lambda_{0,T,\Delta t,N} \geq \lambda_{1,T,\Delta t,N} \geq \cdots \geq \lambda_{N-1,T,\Delta t,N} \geq 0$. The data-driven operator $K_{T,\Delta t, N}$ is understood as an approximation of $K_{T} $ in the following spectral sense:
\begin{itemize}
    \item Let $ \lambda_{j,T,\Delta t,N}$ be a nonzero eigenvalue of $K_{T,\Delta t,N}$. Then, $ \lambda_{j,T,\Delta t,N}$ is employed as an approximation of eigenvalue $ \lambda_{j,T}$ of $K_T$.
    \item Eigenfunction $ \phi_{j,T,\Delta t,N} \in L^2(\mu_N) $ has a continuous representative
        \begin{equation}
            \label{eqVarphiN}
            \varphi_{j,T,\Delta t,N} = \frac{1}{\lambda_{j,T,\Delta t, N}} \int_\Omega k_{T,\Delta t}(\cdot, x)\phi_{j,T,\Delta t,N}(x)\,d\mu_N(x),
        \end{equation}
        defined everywhere on $ \Omega$. The restriction of $ \varphi_{j,T,\Delta t,N}$ to $M$ is a continuously differentiable function, employed as an approximation of $\varphi_{j,T}$ from~\eqref{eqVarphi}.
\end{itemize}
Numerically, the eigenvalues and eigenvectors of $K_{T,\Delta t,N}$ are computed by solving the eigenvalue problem for the $N\times N$ kernel matrix $ \bm K = [ k_{T,\Delta t,N}(x_m, x_n) ]_{mn} / N $, which is the matrix representation of $K_{T,\Delta t,N}$ according to Step~2 above. For kernels with rapidly decaying shape functions (e.g., the Gaussian kernels employed in Section~\ref{secExamples} below), the leading eigenvalues and eigenvectors of $\bm K$ are well approximated by the corresponding eigenvalues and eigenvectors of a sparse matrix obtained by zeroing out small entries of $ \bm K$, considerably reducing computational cost. See, e.g., Appendix~A in \cite{Giannakis19}, or Appendix~B in \cite{DasGiannakis19} for further details on numerical implementation. 

\paragraph{Step 4 (Shift operator)} For each time $ t = q \, \Delta t$, $ q \in \mathbb N_0$, approximate the Koopman operator $U^t : L^2(\mu) \to L^2(\mu)$ by the $q$-step shift operator $U^{q}_N : L^2(\mu_N) \to L^2(\mu_N)$, defined as 
\begin{displaymath}
    U^q_N f(x_n) = 
    \begin{cases}
        f(x_{n+q}), & 0 \leq n \leq N - 1 - q, \\
        0, & n > N - 1 - q.
    \end{cases}
\end{displaymath}
It should be noted that, unlike $U^t f = f \circ \Phi^t$, the shift operator $U^q_N$ is not a composition operator by the underlying dynamical flow---this is because $ \Phi^t$ does not preserve $ \mu_N $-null sets, and thus $ \circ \Phi^t$ does not lift  to an operator on equivalence classes of functions in  $L^2(\mu_N)$. In fact, while $U^t $ is unitary,  $U^q_N $ is a nilpotent operator with $U^N_N = 0$. Still, despite these differences, one can interpret $U^q_N$ as an approximation of the Koopman operator in the following sense: 
\begin{itemize}
    \item Let $\mathcal U^t : C(M) \to C(M)$, $ t \geq 0 $, denote the Koopman operator on continuous functions on the forward-invariant manifold $M$. Let also  $ \iota_N : C(M) \to L^2(\mu_N)$ be the canonical linear operator mapping $C(M)$ functions to their corresponding equivalence classes in $L^2(\mu_N)$, respectively. Then, for any fixed $ q \in \mathbb N_0$ and continuous function $ f \in C(M)$, we have
        \begin{equation}
            \label{eqUEquiv}
            U^q_N \circ \iota_N f = \iota_N \circ \mathcal U^{q\,\Delta t} f + r_N,
        \end{equation}
        where $ r_N \in L^2(\mu_N)$ are residuals whose norm converges to 0, $ \lim_{N\to\infty} \lVert r_N \rVert_{L^2(\mu_N)} = 0$. In contrast, the Koopman operator on $L^2(\mu)$ satisfies $ U^t \circ \iota f = \iota \circ \mathcal U^{t} f$ for any (fixed) $ t \in \mathbb R$, where $ \iota : C(M) \to L^2(\mu)$ is the canonical inclusion map.      
\end{itemize}

\paragraph{Step 5 (Finite-difference operator)} Approximate the generator $ V : D(V) \to L^2(\mu)$ by the finite-difference operator $ V_{\Delta t, N} : L^2(\mu_N) \to L^2(\mu_N)$, where
\begin{displaymath}
    V_{\Delta t, N} = \frac{ U_N - \Id}{ \Delta t}.
\end{displaymath}
Explicitly, we have
\begin{displaymath}
    V_{\Delta t, N} f(x_n) = 
    \begin{cases}
        ( f( x_{n+1} ) - f(x_{n} ) ) / \Delta t, & 0 \leq n \leq N - 2, \\
        -f(x_{N-1}) / \Delta t, & n= N-1.
    \end{cases}
\end{displaymath}
This operator can be understood as an approximation of the generator in the following sense: 
\begin{itemize}
    \item Let $ \mathcal V_{\Delta t} : C(M) \to C(M)$ be the finite-difference approximation of the dynamical vector field $ \mathcal V : C^1(M) \to C(M)$, given by
        \begin{displaymath}
            \mathcal V_{\Delta t} =  \frac{\mathcal U^{\Delta t} - \Id}{\Delta t}. 
        \end{displaymath}
        Then, for any $ f \in C(M)$, we have
        \begin{displaymath}
            V_{\Delta t, N} \circ \iota_N f = \iota_N \circ \mathcal V_{\Delta t} f + r_{\Delta t, N}, 
        \end{displaymath}
        where $\lim_{N\to\infty} \lVert r_{\Delta t, N} \rVert_{L^2(\mu_N)} = 0$. If, in addition, $ f $ lies in $C^1(M)$, then 
        \begin{equation}
            \label{eqFDResidual}
            \mathcal V_{\Delta t} f = \mathcal V f + r_{\Delta t},
        \end{equation}
        where the residual $ r_{\Delta t} $ converges uniformly to 0 as the sampling interval decreases, $ \lim_{ \Delta t \to 0^+} \lVert r_{\Delta t} \rVert_{C(M)} = 0 $. Note that the generator $V$ on $L^2(\mu)$ satisfies $ V \circ \iota f = \iota \circ \mathcal V f$ for any $ f \in C^1(M)$. 
\end{itemize} 

\paragraph{Step 6 (Coherent features)} In order to construct coherent observables analogously to Theorem~\ref{thmMain}, pick two consecutive, nonzero, simple eigenvalues of $K_{T,\Delta t, N}$, which we denote $ \lambda_{T,\Delta t, N}$ and $ \nu_{T,\Delta t, N}$ suppressing $j$ subscripts, and consider corresponding real normalized eigenfunctions $\phi_{T,\Delta t, N}$ and $ \psi_{T,\Delta t, N}$, respectively. Alternatively, a single twofold-degenerate nonzero eigenvalue can be used. Then, form the complex unit vector $ z_{T,\Delta t, N} = ( \phi_{T,\Delta t, N} + i \psi_{T,\Delta t, N} ) / \sqrt 2 \in L^2(\mu_N) $, and compute its continuous representative 
\begin{equation}
    \label{eqZetaN}
    \zeta_{\Delta t, N} = \frac{1}{\sqrt{2}} \int_\Omega k_{T,\Delta t, N}(\cdot, x ) \left( \frac{\phi_{\Delta t, N}}{\lambda_{\Delta t, N}} + i \frac{\psi_{\Delta t, N}}{\nu_{\Delta t, N}}  \right) \, d\mu_N(x).
\end{equation}
The function $ \zeta_{\Delta t, N}$ is employed as a data-driven coherent feature, analogous to $\zeta$ in Corollary~\ref{corMain}. Note, in particular, that $ \zeta_{\Delta t, N}$ is expressible as a finite linear combination of kernel sections $ k(\cdot, x_n ) $, and thus can be empirically evaluated at any point in $ \Omega $. Moreover, we construct data-driven analogs of the autocorrelation function $\alpha_t$ for $ t = q \, \Delta t  $ and the oscillatory frequency $ \omega $ by computing
\begin{equation}
    \label{eqAlphaN}
    \alpha_{q,\Delta t, N} = \langle z_{\Delta t, N}, U^q_{N} z_{\Delta t, N} \rangle_{N}, \quad \omega_{\Delta t, N} = \langle \psi_{\Delta t, N}, V_{\Delta t, N} \phi_{\Delta t, N} \rangle_{N}, 
\end{equation}
respectively.

\subsection{Convergence in the large-data limit}

We are interested in establishing convergence of the data-driven coherent observable $\zeta_{\Delta t, N}$, autocorrelation function $\alpha_{q,\Delta t, N}$, and oscillatory frequency $\omega_{\Delta t, N}$ to their counterparts from Section~\ref{secMain} in a limit of large data, $ N \to \infty$, and vanishing sampling interval, $ \Delta t \to 0$. For that, we follow a similar approach to \cite{DasGiannakis19,GiannakisEtAl19}, who employ spectral approximation results for kernel integral operators by Von Luxburg et al.\ \cite{VonLuxburgEtAl08}. The principal elements of this approach are as follows.

\paragraph{Operators on continuous functions} Since the operators $K_T$ and $K_{T,\Delta t, N}$ act on different Hilbert spaces, we use the space of continuous functions on the forward-invariant manifold $M$ as a universal comparison space to establish spectral convergence. In particular, since the kernels $k_T$ and $k_{T,\Delta t, N}$ are all continuous, one can consider integral operators $ \mathcal K_T : C(M) \to C(M)$ and $ \mathcal K_{T,\Delta t, N} : C(M) \to C(M)$, defined analogously to $K_T : L^2(\mu) \to L^2(\mu)$ and $K_{T,\Delta t, N} : L^2(\mu_N) \to L^2(\mu_N)$, respectively. We then have $ \iota \circ \mathcal K_T = K_T \circ \iota $ and $ \iota_N \circ \mathcal K_{T,\Delta t, N} = K_{T,\Delta t, N} \circ \iota_N$, and it is straightforward to verify that $ \lambda_{j,T}$ (resp.\ $\lambda_{j,T,\Delta t, N}$) is a nonzero eigenvalue of $K_T$ (resp.\ $K_{T,\Delta t, N}$) if and only if it is a nonzero eigenvalue of $ \mathcal K_T$ (resp.\ $\mathcal K_{T,\Delta t, N}$). Moreover, if $ \phi_{j,T} \in L^2(\mu)$ (resp.\ $ \phi_{j,T,\Delta t, N} \in L^2(\mu_N)$) is a corresponding eigenfunction of $K_T$ (resp.\ $K_{T,\Delta t,N}$), then $ \varphi_{j,T} \in C(M)$ from~\eqref{eqVarphi} (resp.\ $ \varphi_{j,T,\Delta t, N} \in C(M)$ from~\eqref{eqVarphiN}) is a corresponding eigenfunction of $\mathcal K_T$ (resp.\ $\mathcal K_{T,\Delta t, N}$). It can further be shown that $\mathcal K_T$ is compact, and clearly $\mathcal K_{T,\Delta t, N}$ has finite rank.     

\paragraph{Ergodicity and physical measures} Let $B_\mu \subseteq \Omega $ be the basin of the ergodic invariant measure $\mu $ in $ M $, i.e., the set of initial conditions $ x_0 \in \Omega$ such that the corresponding sampling measures $ \mu_N $ weak-converge to $ \mu $,
\begin{equation}
    \label{eqWeakConv}
    \lim_{N\to\infty} \mathbb E_{\mu_N} f = \mathbb E_\mu f, \quad \forall f \in C_b(\Omega),
\end{equation}
for Lebesgue almost every sampling interval $ \Delta t $. Here, $\mathbb E_\rho f = \int_\Omega f \, d \rho$ denotes expectation with respect to a measure $ \rho $, and $C_b(\Omega)$ is the Banach space of continuous, real-valued functions on $ \Omega $ equipped with the uniform norm. By ergodicity of the dynamical flow $ \Phi^t$, $ B_\mu \cap X $ is a dense subset of the support $X $ of $ \mu $. Moreover, for a class of dynamical systems possessing so-called physical measures \cite{Young02} the basin $B_\mu$ has positive measure with respect to an ambient probability measure on state space $ \Omega $ from which initial conditions are drawn, even if $X$ is a null set with respect to that measure. In such situations, the data-driven scheme described in Section~\ref{secDataDrivenDesc} converges from a ``large'' set of experimentally accessible initial conditions, which need not lie on the support of $\mu$. Examples include the L63 system, where the the ergodic invariant measure supported on the Lorenz attractor is a Sinai-Ruelle-Bowen (SRB) measure with a basin of positive Lebesgue measure in $ \Omega = \mathbb R^3$ \cite{Tucker99}. For simplicity of exposition, and without loss of generality with regards to asymptotic convergence, we will henceforth assume that the initial state $x_0$ lies in $ B_\mu \cap M $. Moreover, $ \Delta t \to 0 $ limits will be assumed to be taken along a sequence such that~\eqref{eqWeakConv} holds.  

\paragraph{Spectral convergence} Since our approach for coherent feature extraction employs on eigenvalues and  eigenvectors of kernel integral operators, it is necessary to ensure that the family $ \mathcal K_{T,\Delta t, N}$ converges to $ \mathcal K_T$ in a sufficiently strong sense so as to imply spectral convergence. Here, we consider the iterated limit of $N \to \infty$ followed by $ \Delta t \to 0$; under the former limit, empirical expectation values with respect to the sampling measures converge to expectation values with respect to the invariant measure (according to~\eqref{eqWeakConv}), and under the latter limit the kernels based on discrete-time delay-coordinate maps converge to their continuous-time counterparts (according to~\eqref{eqDTConvergence}). In particular, we have:

\begin{proposition}
    \label{propSpecConv}With notation and assumptions as above, let $ \lambda_{j,T} $ be a nonzero eigenvalue of $ \mathcal K_T$, where the ordering $ \lambda_{0,T} \geq \lambda_{1,T} \geq \cdots$  is in decreasing order and includes multiplicities. Let $ \Pi_{j,T} : C(M) \to C(M)$ be the spectral projection to the corresponding eigenspace. Then, the following hold:
    \begin{enumerate}[(i),wide]
        \item The $j$-th eigenvalues $ \lambda_{j,T,\Delta t, N}$ of $ \mathcal K_{T,\Delta t, N}$ (ordered with the same convention as the eigenvalues of $ \mathcal K_T$) converge to $ \lambda_{j,T}$, in the sense of the iterated limit
            \begin{displaymath}
                \lim_{\Delta t \to 0} \lim_{N\to\infty} \lambda_{j,T,\Delta t,N} = \lambda_{j,T}.
            \end{displaymath}
        \item For any neighborhood $ \Sigma \subseteq \mathbb C$ such that $ \sigma( \mathcal K_T ) \cap \Sigma = \{ \lambda_{j,T} \} $, the spectral projections $ \Pi_{\Sigma,T,\Delta t,N}$ of $\mathcal K_{T,\Delta t, N}$ onto $ \Sigma $ converge strongly to $ \Pi_{j,T}$. In particular, for any eigenfunction $ \varphi_{j,T} \in C(M)$ of $ \mathcal K_T$ corresponding to eigenvalue $ \lambda_{j,T} $ there exist eigenfunctions $ \varphi_{j,T,\Delta t, N} \in C(M)$ of $\mathcal K_{T,\Delta t, N}$ corresponding to $ \lambda_{j,T,\Delta t, N}$, such that
            \begin{displaymath}
                \lim_{\Delta t \to 0} \lim_{N\to\infty} \lVert \varphi_{j,T,\Delta t, N} - \varphi_{j,T } \rVert_{C(M)} = 0.
            \end{displaymath}
    \end{enumerate}
\end{proposition}

\begin{remark} \label{rkKernel2} Analogous spectral convergence results to Proposition~\ref{propSpecConv} hold for integral operators with data-dependent kernels $k_{T,\Delta t, N}$, so long as these kernels have well defined $N \to \infty$ limits in $C(M)$ norm. Examples of such kernels include Markov-normalized kernels \cite{BerrySauer16,CoifmanLafon06,CoifmanHirn13} and variable-bandwidth Gaussian kernels \cite{BerryHarlim16}. See, e.g., Theorem~7 in \cite{GiannakisEtAl19} for a  a spectral convergence result for data-dependent kernels related to the kernels employed in the numerical experiments in Section~\ref{secExamples}. 
\end{remark}

A corollary of Proposition~\ref{propSpecConv} is that the properties the data-driven coherent observable $\zeta_{\Delta t, N}$ from~\eqref{eqZetaN} and the corresponding empirical autocorrelation function and oscillatory frequency in~\eqref{eqAlphaN} converge to their counterparts from Theorem~\ref{thmMain}, and thus obey the same pseudospectral bounds associated with dynamical coherence.

\begin{corollary}
    Under the assumptions of Proposition~\ref{propSpecConv}, the following hold in the large-data limit, $ \Delta t \to 0$ after $N \to \infty$, where $ \zeta$, $ \alpha_t$, and $\omega$ are defined in Theorem~\ref{thmMain}:  
    \begin{enumerate}[(i), wide]
        \item $ \zeta_{\Delta t, N}$ converges to the coherent feature $\zeta$, uniformly on the forward-invariant manifold $M$, i.e., 
        \begin{displaymath}
            \lim_{\Delta t \to 0 } \lim_{N\to \infty} \lVert \zeta_{\Delta t, N} - \zeta \rVert_{C(M)} = 0. 
        \end{displaymath}
        \item For any $ q \in \mathbb N$, the empirical autocorrelation $ \alpha_{q,\Delta t, N}$ converges to the autocorrelation function $ \alpha_t $  at $ t = q \, \Delta t $.
        \item The empirical oscillatory frequency $\omega_{\Delta t, N}$ converges to the frequency $\omega$.   
    \end{enumerate}
\end{corollary}

\begin{proof}
    The uniform convergence of $ \zeta_{\Delta t, N}$ to $\zeta$ in Claim~(i) is a direct consequence of Proposition~\ref{propSpecConv}. Claim~(ii) follows from the Claim~(i), in conjunction with the residual estimate in~\eqref{eqUEquiv}, viz.
    \begin{align*}
        \lim_{\Delta t \to 0} \lim_{N \to \infty} \alpha_{q,\Delta t, N} &= \lim_{\Delta t \to 0}\lim_{N\to \infty} \langle z_{\Delta t, N}, U^q_N z_{\Delta t, N} \rangle_N \\
        &= \lim_{\Delta t \to 0}\lim_{N\to \infty} \langle \iota_N \zeta_{\Delta t, N}, U^q_N \iota_N \zeta_{\Delta t, N} \rangle_N \\
        &= \lim_{\Delta t \to 0}\lim_{N\to \infty} \langle \iota_N \zeta_{\Delta t, N}, \iota_N \mathcal U^{q\, \Delta t} \zeta_{\Delta t, N} + r_N \rangle_N \\
        &= \lim_{\Delta t \to 0}\lim_{N\to \infty} \langle \iota_N \zeta_{\Delta t, N}, \iota_N \mathcal U^{q\, \Delta t} \zeta_{\Delta t, N}\rangle_N \\
        &= \lim_{\Delta t \to 0 } \lim_{N\to\infty} \frac{1}{N} \sum_{n=0}^{N-1} \zeta_{\Delta t, N}(x_n) z_{\Delta t, N}(x_{n+q}) \\
        &= \int_\Omega \zeta^* \mathcal U^{q\,\Delta t} \zeta \, d\mu = \langle \iota \zeta, \iota \mathcal U^{q\,\Delta t} \zeta \rangle = \langle z, U^{q\,\Delta t} z \rangle \\
        &= \alpha_{q\, \Delta t}.
    \end{align*}
    Claim~(iii) follows similarly, using a finite-difference residual estimate in~\eqref{eqFDResidual}, in conjunction with the $C^1$-norm convergence of $k_{T,\Delta t}$ to $k_T$ as $\Delta t \to 0$ (see~\eqref{eqDT1Convergence}).  
\end{proof}
%As in Section~\ref{secMain}, we oftentimes suppress $j$ subscripts from our notation for $ \lambda_{j,T,Q,N}$ and $ \phi_{j,T,Q,N}$ when our interest is on specific eigenvalue/eigenvector pairs.     

%Approximate the generator $ V : D(V) \to L^2(\mu)$ by the skew-adjoint, finite-difference operator $ V_{\Delta t, N} : L^2(\mu_N) \to L^2(\mu_N)$, where
%\begin{displaymath}
%    V_{\Delta t, N} = \frac{ U_N - U^*_N}{ 2 \, \Delta t}.
%\end{displaymath}
%Note that 
%\begin{displaymath}
%    V_{\Delta t, N} f(x_n) = 
%    \begin{cases}
%        f(x_1) / (2 \, \Delta t), & n = 0, \\
%        ( f( x_{n+1} ) - f(x_{n-1} ) ) / ( 2\, \Delta t), & 1 \leq n \leq N - 2, \\
%        -f(x_{N-1}) / (2\,\Delta t), & n= N-1.
%    \end{cases}
%\end{displaymath}

\subsection{\label{secKernels}Choice of kernel}

Following \cite{DasEtAl20}, in the numerical experiments described below we employ integral operators $K_{T,\Delta t, N} : L^2(\mu_N) \to L^2(\mu_N)$ associated with a family of symmetric, Markov-normalized  kernels $k_{T,\Delta t, N}$ constructed using the variable-bandwidth Gaussian kernels in conjunction with the bistochastic Markov normalization procedure  proposed in \cite{BerryHarlim16} and \cite{CoifmanHirn13}, respectively. Specifically, to build $k_{T,\Delta t, N}$ we start from a radial Gaussian kernel $\bar k_{T,\Delta t} : \Omega \times \Omega \to \mathbb R_+$ on delay-coordinate mapped data, 
\begin{displaymath}
    \bar k_{T,\Delta t}( x, x' ) = \exp\left( - \frac{d^2_{T,\Delta t}(x,x')}{\bar \sigma^2} \right),
\end{displaymath}
where $\bar \sigma $ is a positive bandwidth parameter determined numerically from the data (see, e.g., Algorithm~1 in \cite{Giannakis19}). Using this kernel, we compute the bandwidth function $ \rho_{T,\Delta t, N} \in C(\Omega)$ given by 
\begin{displaymath}
    \rho_{T,\Delta t, N}(x) = \left( \int_\Omega \bar k_{T,\Delta t}(x,\cdot) \, d\mu_N \right)^{-1/  m} = \left( \frac{1}{N} \sum_{n=0}^{N-1} \bar k_{T,\Delta t}(x,x_n)  \right)^{-1/m}.
\end{displaymath}
Here, $ m > 0 $ is an estimate of the dimension of the support $X$ of the invariant measure, computed through the same procedure used to tune the kernel bandwidth $ \bar \sigma $. We then build the variable-bandwidth kernel $ \kappa_{T,\Delta t, N} : \Omega \times \Omega \to \mathbb R_+$, where
\begin{equation}
    \label{eqKVB}
    \kappa_{T,\Delta t, N}(x,x') = \exp \left( -  \frac{d^2_{T,\Delta t}(x,x')}{\sigma^2 \rho_{T,\Delta t, N}(x) \rho_{T,\Delta t, N}(x')} \right).
\end{equation}
In the above, $ \sigma $ is a positive bandwidth parameter determined automatically in a similar manner as $ \bar \sigma $, though note that in general $ \sigma $ and $ \bar \sigma $ have different values. 

By construction, $ \kappa_{T,\Delta t, N}$ is, continuous, positive, and bounded away from zero on $M \times M$. Intuitively, the function $ \rho_{T,\Delta t, N}^{-m}$ can be thought of as a kernel estimate of the ``sampling density'' of the data relative to an ambient measure. The variable-bandwidth construction in~\eqref{eqKVB} can then be thought of as a data-adaptive adjustment of the bandwidth $ \sigma $, such that a data point $x$ is assigned a smaller (larger) bandwidth $ \sigma \rho_{T,\Delta t, N}(x)$ when the sampling density is higher (lower), thus reducing sensitivity to sampling errors. This intuition can be made precise if the support $ X $ has the structure of a Riemannian manifold and $ \mu $ the structure of a smooth volume form. In that case, the variable-bandwidth kernel effects a conformal change of Riemannian metric on the data such that in the new geometry the invariant measure has constant density relative to the Riemannian volume form; see \cite{Giannakis19} for further details. 

Next, we normalize the kernel $\kappa_{T,\Delta t, N} $ to obtain a symmetric Markov kernel $ k_{T,\Delta t, N} : \Omega \times \Omega \to \mathbb R_+$ by first computing the strictly positive, continuous functions 
\begin{displaymath}
    u_{T,\Delta t, N} = \int_\Omega \kappa_{T,\Delta t, N}(\cdot, x ) \, d\mu_N(x), \quad v_{T,\Delta t, N} = \int_\Omega \frac{\kappa_{T,\Delta t, N}(\cdot, x)}{u_{T,\Delta t, N}(x)} \, d\mu_N(x),
\end{displaymath}
and then defining 
\begin{equation}
    \label{eqKBistoch}
    k_{T,\Delta t, N}(x,x') = \int_\Omega \frac{\kappa_{T,\Delta t, N}(x,x'')\kappa_{T,\Delta t,N}(x'',x)}{ u_{T,\Delta t, N}(x) v_{T,\Delta t, N}(x'') u_{T,\Delta t, N}(x')} \, d\mu_N(x'').
\end{equation}
It can be readily verified that with this definition $k_{T,\Delta t, N}$ is a symmetric, strictly positive kernel with the Markov property, $ \int_{\Omega} k_{T,\Delta t, N}(x,x') \, d\mu_N(x') = 1$, for all $ x \in \Omega$. Moreover, $k_{T, \Delta t, N}$ is (strictly) positive-definite on the support of $ \mu_N $ if $ \kappa_{T,\Delta t, N}$ is (strictly) positive-definite. It can further be shown \cite{GiannakisEtAl19} that in the large-data limit, $ \Delta t \to 0 $ after $N \to \infty$, $ k_{T,\Delta t,N}$, converges to an $L^2(\mu)$-Markov, symmetric, continuous kernel $k_{T}$ so an analogous spectral convergence result to Proposition~\ref{propSpecConv} holds for this class of kernels (see also Remarks~\ref{rkKernel1} and~\ref{rkKernel2}). 

For the purposes of extraction of coherent observables of measure-preserving, ergodic dynamics, symmetric Markov kernels have the natural property of exhibiting a constant eigenfunction corresponding to the top eigenvalue, $ \lambda_{0,T,\Delta t, N} = \lambda_{0,T} = 1$, with the remaining eigenfunctions capturing mutually orthogonal features orthogonal to the constant. See Appendix~B in \cite{DasEtAl20} for pseudocode for solving the eigenvalue problem for $K_{T,\Delta t, N}$, where explicit formation of the kernel in~\eqref{eqKBistoch} is avoided through singular value decomposition of a non-symmetric kernel matrix.

\section{\label{secExamples}Numerical examples}

\subsection{\label{secDatasets}Dataset description}

As an application of the results in Sections~\ref{secMain}--\ref{secDataDriven}, we study the properties of eigenfunctions of the integral operators $K_{T,\Delta t, N}$ induced by the L63 system on $  \Omega = \mathbb R^3$ with the standard parameters, 
\begin{gather*}
    \dot x = \vv V(x), \quad x = (x^1,x^2,x^3) \in \mathbb R^3, \quad \vv V(x) = ( V^1, V^2, V^3 ), \\
    V^1 = 10( x^2 - x^1 ), \quad V^2 = 28 x^1 - x^2 - x^1 x^3, \quad V^3 = x^1 x^2 - 8 x^3 / 3.
\end{gather*}
We generate numerical trajectories $x_0, \ldots, x_{\tilde N-1} \in \Omega$  sampled at an interval $ \Delta t = 0.01 $ natural time units using Matlab's \texttt{ode45} solver. Numerical integration starts at an arbitrary point $ \tilde x \in \mathbb R^3$, and we allow the state to settle near the Lorenz attractor over a spinup time of $ 640 $ time units before collecting the first sample $x_0$. In anticipation of the fact that we will be using the delay-coordinate map in~\eqref{eqFQ}, we sample a total of $ \tilde N = N + Q - 1$ states, where $ Q $ is the number of delays, and $N $ is fixed at $N = \text{64,000} $.  

We consider two experiments, one with $Q=1 $ corresponding to no delays ($T=0$) and another one with $ Q = 800$ corresponding to a delay-embedding window of $ T = Q \, \Delta t = 8 $ natural time units. The latter, is approximately equal to 9 Lyapunov characteristic times $ T_L = 1 / \Lambda$, where $ \Lambda \approx 0.91 $ \cite{Sprott03} is the positive Lyapunov exponent of the L63 system. The $ T = 8 $ embedding window is also approximately equal to 10 oscillations assuming a characteristic oscillatory timescale of $ T_o = 0.8 $. In both cases we set the observation map $ F : \Omega \to Y $ to the identity map on $  \mathbb R^3 $, so the corresponding delay coordinate map $F_{Q,\Delta t} $ takes values in $ Y^Q = \mathbb R^{3 Q} $. Note that, after delay embedding, each experiment has $ N = \text{64,000} $ samples $ y_n = F_{Q,\Delta t}$ available for analysis, which corresponds to $ 800 $ oscillatory timescales $T_o$. 

As stated in Section~\ref{secSetup}, this L63 setup rigorously satisfies all the assumptions made in Theorem~\ref{thmMain} \cite{Tucker99,LuzzattoEtAl05,LawEtAl14}. In addition, since $ \Delta t \ll T_o$, $ ( N - 1 ) \, \Delta t \gg T_o$, and in the $T = 8$ setup, $ \Delta t \ll T $, we expect no significant sampling errors to be present in our numerical experiments; in particular, we expect the leading eigenfunctions of the data-driven integral operators $K_{T,\Delta t, N}$ to be good approximations of the corresponding eigenfunctions of the operators $K_T$ from Theorem~\ref{thmMain}.

\subsection{Coherent observables}

We now discuss the properties of eigenfunctions of $K_{T,\Delta t, N}$ constructed using the approach described in Section~\ref{secKernels}, some of which were already shown in Figures~\ref{figPhi} and~\ref{figEta}. All results were obtained using the symmetric Markov kernels in~\eqref{eqKBistoch} with $N = \text{64,000}$, $ \Delta t = 0.01$, and representative values of $T$ in the range 0 to 64. For the rest of this section, we suppress $ \Delta t $ and $N $ indices from our notation. Moreover, we do not distinguish between eigenfunctions $ z \in L^2(\mu_N)$ and their continuous representatives $ \zeta \in C(M)$, as our visualizations will be restricted to the training dataset $\{ x_n \}_{n=0}^{N-1}$ for which $ z(x_n) = \zeta(x_n)$.  

We begin in Figure~\ref{figLambda} with a plot of the leading 20 eigenvalues $ \lambda_{j,T}$ of $K_T $ for $T = 0 $ and $ T=8 $, where both operators have the top eigenvalue  $ \lambda_{0,T} = 1 $ by Markovianity of the kernels. When $T = 0 $, $ K_T $ has a small spectral gap $ \lambda_{0,T} - \lambda_{1,T} \approx 0.0075$, and the subsequent eigenvalues exhibit a gradual decay, reaching $ \lambda_{j,T} \approx 0.8$ at $ j = 20$. In contrast, when $ T = 8 $, $K_T$ exhibits a significantly larger spectral gap $ \lambda_{0,T} - \lambda_{1,T} \approx 0.4$, with a nearly degenerate corresponding eigenspace. In particular, we have $ \lambda_{1,T} \approx 0.6032$ and  $ \lambda_{2,T} \approx 0.6015 $, and the corresponding gap parameters from Theorem~\ref{thmMain} take the values $ \delta_T \approx 0.001$ and $ \tilde \delta_T \approx 0.002 $. Thus, on the basis of Theorem~\ref{thmMain}, the complex-valued observable $ z = ( \phi_{1,T} + i \phi_{2,T} ) / \sqrt{2}$ for $T = 8 $ is a good candidate of a dynamically coherent feature evolving as an $ \epsilon$-approximate eigenfunction of the Koopman operator with small $\epsilon$. The scatterplots and time series plots in Figure~\ref{figPhi} were already suggestive of this behavior, which we now examine in further detail. As a point of comparison, we consider the corresponding observable $z$ constructed from the leading eigenfunctions of $K_T$ at $T = 0 $, which were also depicted in Figure~\ref{figPhi}. 

\begin{figure}
    \centering
    \includegraphics[width=.7\linewidth]{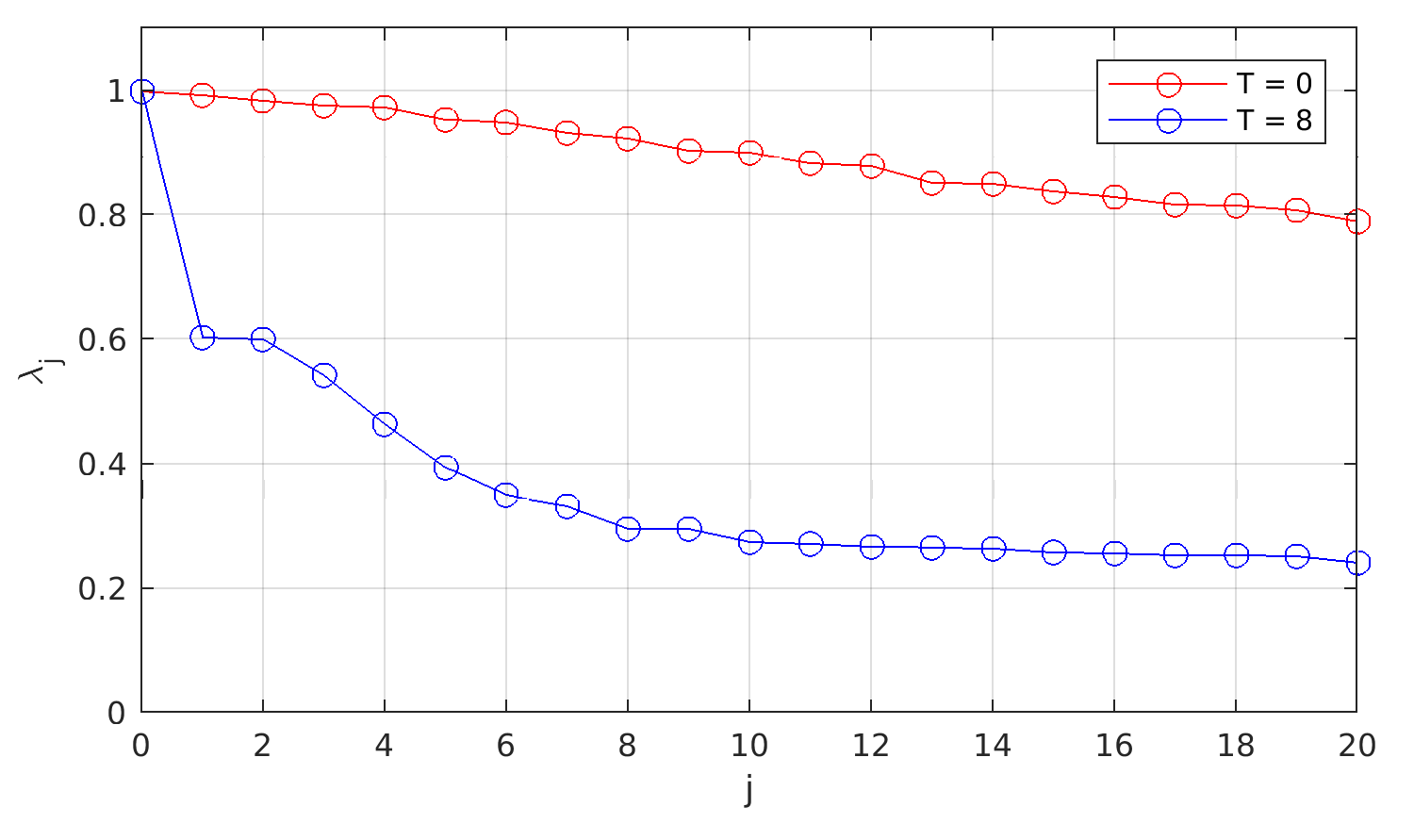}
    \caption{\label{figLambda}Leading 20 eigenvalues $\lambda_{j,T}$ of the integral operators $K_T$ for $T=0 $ and $ T = 8 $.}
\end{figure}

Figure~\ref{figZ} shows the evolution of the observables $z$ as a time-parameterized curve $ t_n \mapsto z(x_n)$ on the complex plane over a portion of the training data spanning 50 natural time units. In effect, these plots correspond to samplings of complex-valued functions on the Lorenz attractor along dynamical trajectories, akin to the time series plots in Figure~\ref{figPhi} which (up to a scaling by a factor of $\sqrt{2}$) correspond to the real and imaginary parts of $z$. The $T=0$ evolution traces out what qualitatively resembles a two-dimensional projection of the attractor. In particular, we do not expect a dynamically coherent behavior for this observable, as its evolution comprises of two cycles with a mixing region when $\Real z \simeq 0 $ which is not too different from the raw L63 dynamics. This lack of coherence is manifestly visible in the scatterplots in Figure~\ref{figPhi} depicting the real and imaginary parts of $ z $ acted upon by the Koopman operator, and can also be assessed more quantitatively through plots of the time-autocorrelation function $\alpha_t $, shown in Figure~\ref{figAlpha}. There, the modulus $ \lvert \alpha_t \rvert$ is seen to rapidly decay from its initial value $\lvert \alpha_0 \rvert = 1$, reaching $ \lvert \alpha_t \rvert \approx 0.05$ at $ t \approx 0.3$, and never exceeds 0.4 after $ \simeq 1 $ Lyapunov time.           

In contrast, the observable $z $ constructed from the eigenfunctions of $ K_T$ at $ T = 8$ exhibits a fundamentally different behavior, consistent with an approximate cycle that remains coherent over several Lyapunov timescales. In Figure~\ref{figZ}, the $ T = 8 $ dynamical trajectory lies in what appears to be a disk in the complex plane, executing a predominantly azimuthal motion with a slow radial motion (amplitude modulation) superposed. In particular, the real and imaginary parts of $ z $ have a 90$^\circ$ phase difference to a good approximation (at least when $ \lvert z \rvert $ is not too small), and as indicated by the time series plots in Figure~\ref{figPhi}, they have a nearly constant characteristic frequency. The coherent dynamical evolution stemming from this behavior is visually evident in the scatterplots of the real and imaginary parts of $U^t z $ in Figure~\ref{figPhi}, which appear to ``resist'' mixing of level sets on significantly longer timescales than the $ T = 0 $ eigenfunctions. 

More quantitatively, in Figure~\ref{figAlpha}, the evolution of the autocorrelation function $\alpha_t$ of $z$ for  $ T= 8 $ is consistent with an amplitude-modulated harmonic oscillator with a well-defined carrier frequency and slowly-varying envelope function. In particular, the real and imaginary parts of $ \alpha_t$ oscillate at a near-constant frequency, and remain phase-locked to a 90$^\circ$ phase difference at least out to $ t = 10 $ natural time units, or $ \simeq 10 $ Lyapunov times. Meanwhile, the modulus $ \lvert \alpha_t \rvert$ exhibits a significantly slower decay than what was observed for $ T = 0 $, and remains above 0.4 for all $ t \in [ 0, 10 ]$.  In Figure~\ref{figOmega} we compare the evolution of the autocorrelation function $\alpha_t$ with a pure sinusoid $ e^{i\omega t}$ with frequency $ \omega $ determined through the finite-difference-approximated generator using~\eqref{eqAlphaN}. The generator-based frequency, $ \omega \approx 8.24$ (corresponding to a period of $ 2 \pi / \omega \approx 0.76$), is seen to accurately capture the carrier frequency of the $ \alpha_t $ signal, as expected from Theorem~\ref{thmMain}, with a slow build-up of phase decoherence that becomes noticeable by $ t \simeq 10 $. 

\begin{figure}
    \includegraphics[width=\linewidth]{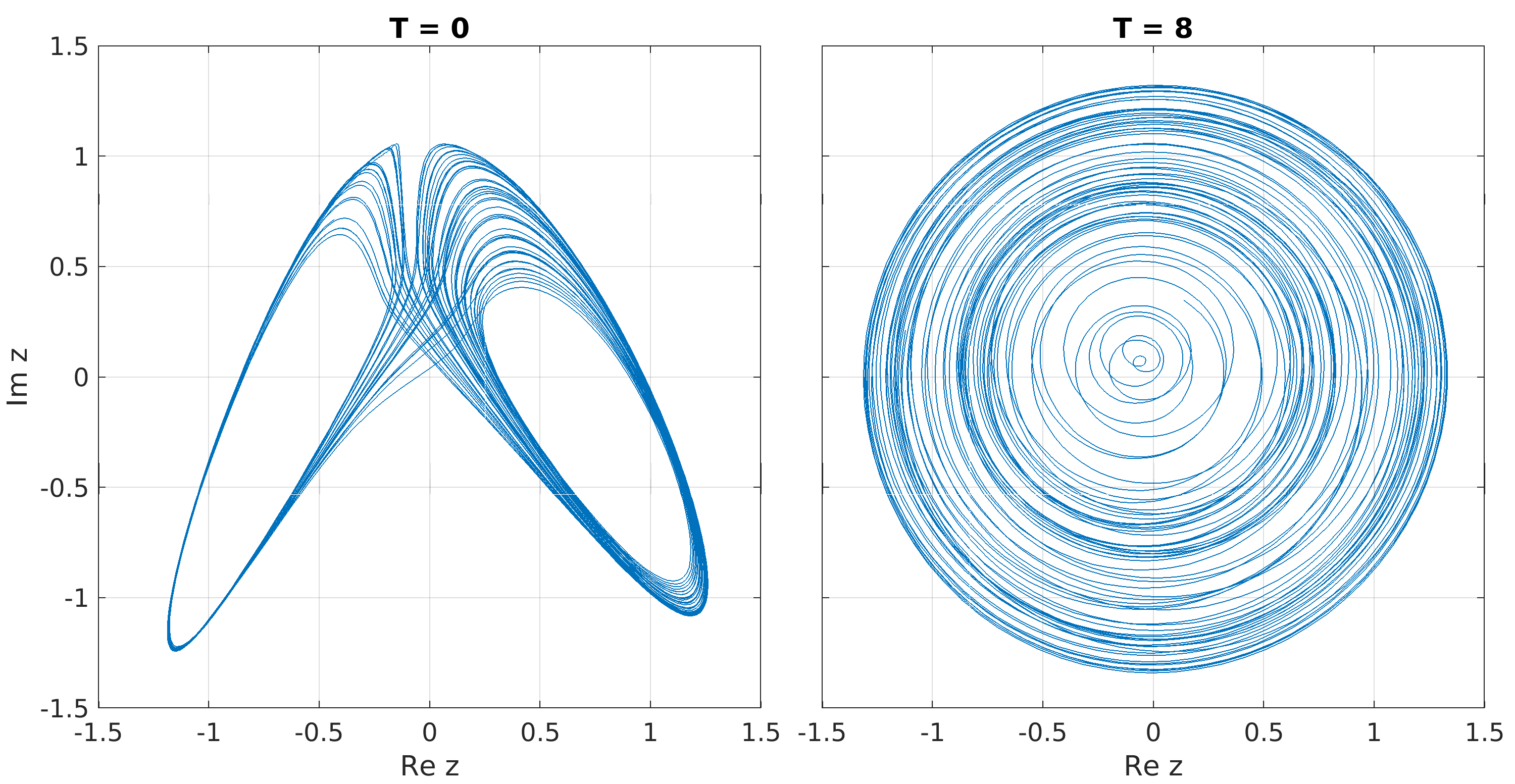}
    \caption{\label{figZ}Evolution of the real and imaginary parts of the observable $ z = ( \phi_{1,T} + i \phi_{2,T} ) / \sqrt{2} $, constructed using the leading two nonconstant eigenfunctions of the integral operator $K_T$ for no delays ($T=0$) and $T=8$. Here, $z$ is plotted as a time-parameterized curve $ t_n \mapsto z(x_n)$ on the complex plane, corresponding to a sampling of its values along an L63 dynamical trajectory at times $ t_n = n \, \Delta t$. For clarity of visualization, $t_n $ is restricted to a time interval of length 50 (whereas the full training datasets span 640 natural time units).} 
\end{figure}

\begin{figure}
    \centering
    \includegraphics[width=.7\linewidth]{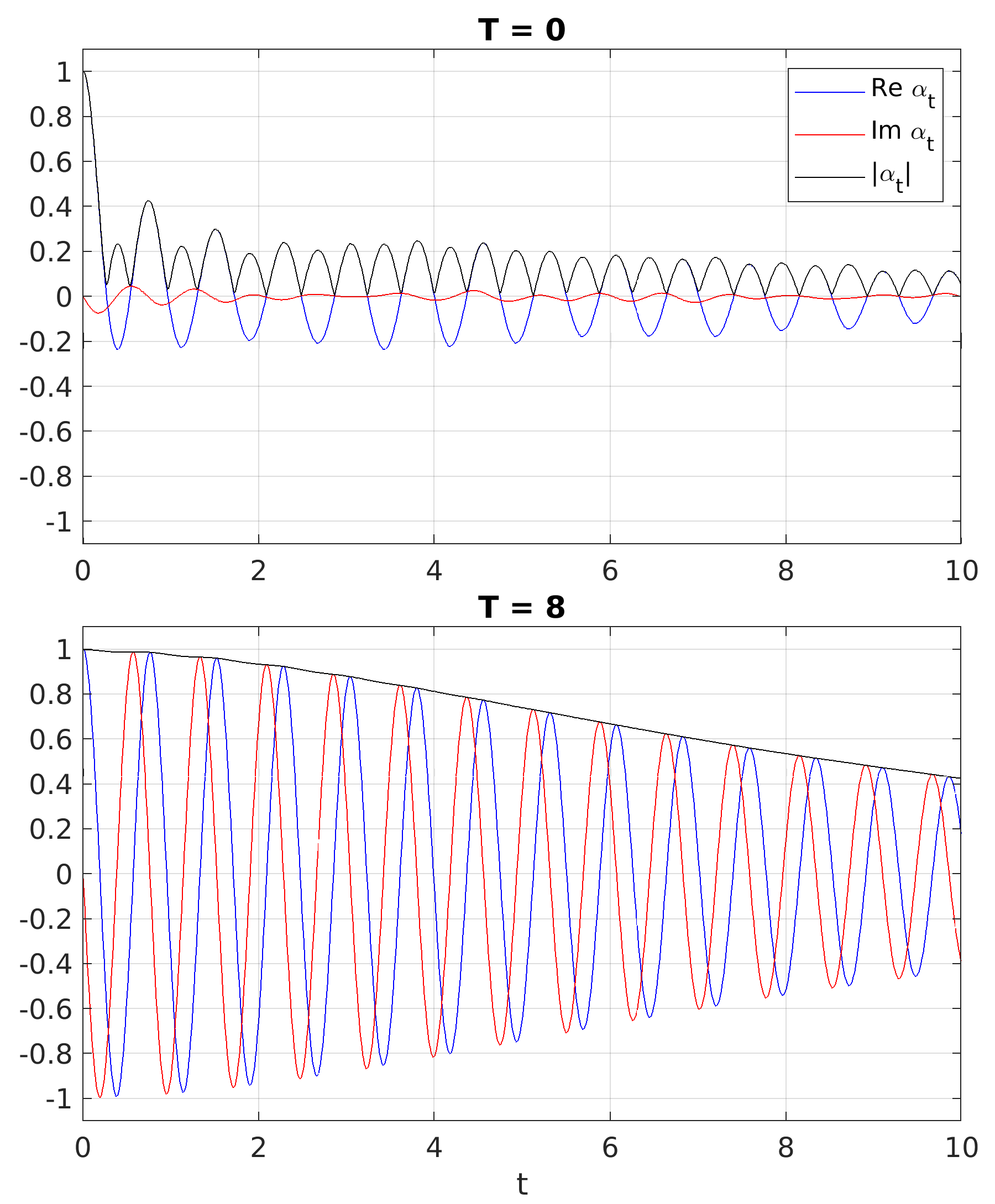}
    \caption{\label{figAlpha}Real part, imaginary part, and modulus of the time autocorrelation function $ \alpha_t$ of the observables $z $ in Figure~\ref{figZ}.}
\end{figure}

\begin{figure}
    \centering
    \includegraphics[width=.7\linewidth]{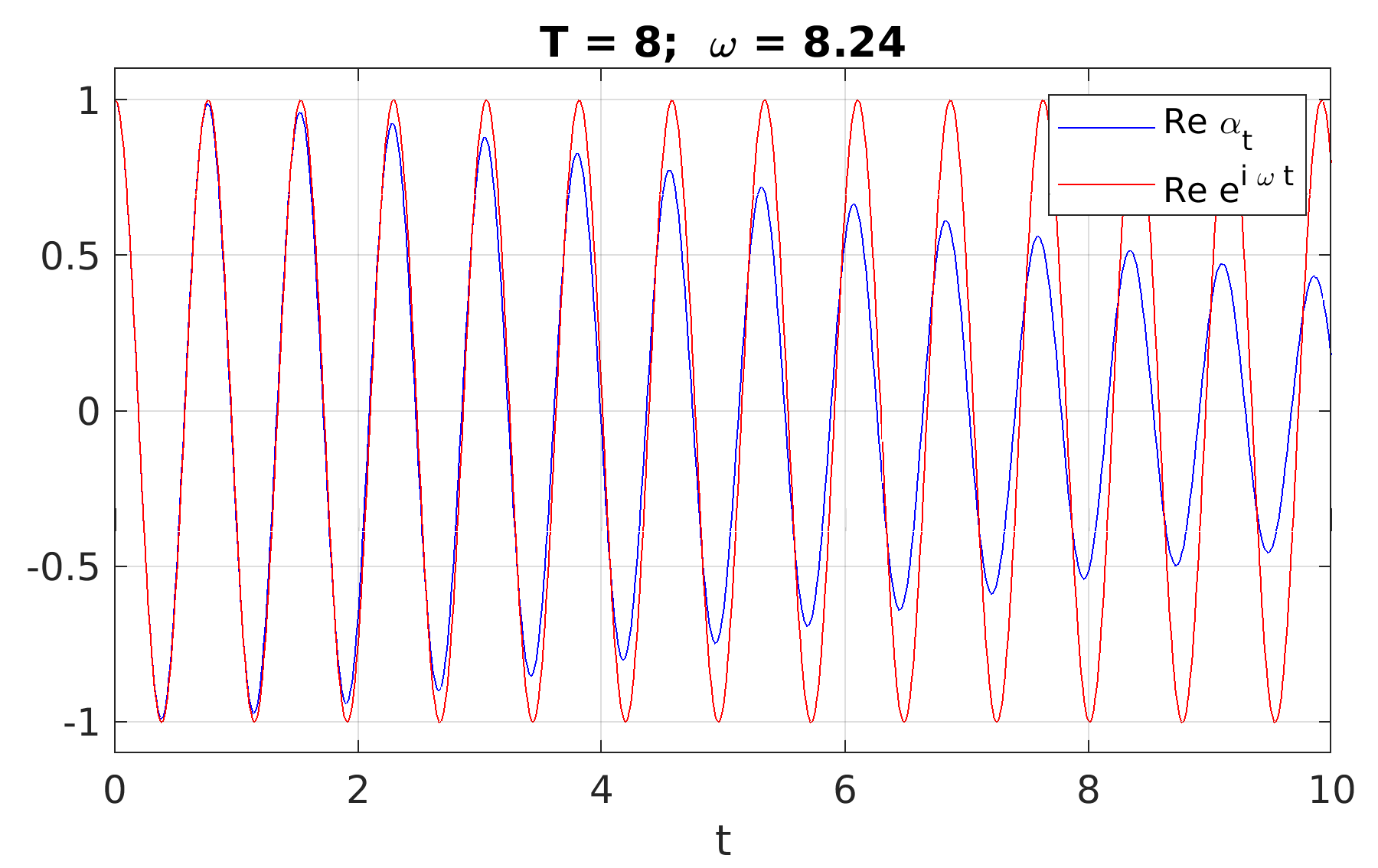}
    \caption{\label{figOmega} A comparison of the real part of the autocorrelation function $ \alpha_t $ with a pure cosine wave $ \cos\omega t = \Real e^{i\omega t} $ for the $T=8$ coherent observable from Figure~\ref{figZ}. The frequency $ \omega $ was computed through~\eqref{eqAlphaN} using the finite-difference approximation of the generator.}
\end{figure}

Intriguingly, the $ \omega \approx 8.24$ frequency identified here through eigenfunctions of $K_T $ is close to an $ 8.18$ approximate eigenfrequency identified in \cite{DasEtAl20} through spectral analysis of a compact approximation to the generator $V$ constructed using reproducing kernel Hilbert space (RKHS) techniques. The RKHS-based eigenfrequency has a corresponding approximate Koopman eigenfunction, $ z_{\text{RKHS}}$, which has a qualitatively similar spatial structure on the L63 attractor as the approximate eigenfunction $z$ identified here (compare Figure~5 in \cite{DasEtAl20} with Figure~\ref{figPhi} of this paper). Moreover, both $z$ and $z_\text{RKHS}$ resemble an observable identified by Korda et al.\ \cite{KordaEtAl20} through a spectral analysis technique for Koopman operators utilizing Christoffel-Darboux kernels in frequency space (see Figure 13 in \cite{KordaEtAl20}). Having been identified via three independent data analysis techniques, it thus appears that the approximate eigenfrequency $\omega \simeq 8.2 $ and the corresponding approximate eigenfunction with the structure depicted in Figure~\ref{figPhi} are robust features of the L63 system, warranting further investigation. 

\section{\label{secConclusions}Concluding remarks}

In this paper we have studied how kernel integral operators constructed from delay-coordinate mapped data can identify, through their eigenfunctions, dynamically coherent features of measure-preserving, ergodic dynamical systems. We have shown that a class of eigenfunctions of such operators lead to complex-valued observables with an approximately cyclical evolution, behaving as $ \epsilon $-approximate eigenfunctions of the Koopman operator for a bound $ \epsilon $ that decreases with the length of the embedding window. Such observables encapsulate a natural notion of dynamical coherence, so we have argued,  in the sense of having high regularity on the attractor, a well-defined oscillatory frequency, and a slowly decaying time-autocorrelation amplitude. In addition, the spectral bounds were  explicitly characterized as functions of the embedding window length, evolution time, and appropriate spectral gap parameters. 

These results extend previous work on integral operators approximating the point spectrum of the Koopman operator in the infinite-delay limit \cite{DasGiannakis19,Giannakis19} to the setting of mixing dynamical systems with continuous Koopman spectra. Thus, they provide a theoretical interpretation of the efficacy of a number of data-driven techniques utilizing delay embeddings, including DMDC \cite{BerryEtAl13}, HAVOK  analysis \cite{BruntonEtAl17}, NLSA \cite{GiannakisMajda11c,GiannakisMajda12a,GiannakisMajda13}, and SSA \cite{BroomheadKing86,VautardGhil89}, in extracting coherent signals from complex systems. An attractive aspect of these methods is that they are amenable to consistent data-driven approximation from time series data based on techniques originally developed in the context of spectral clustering \cite{VonLuxburgEtAl08}. In particular, the data-driven schemes are rigorously applicable in situations where the invariant measure is supported on non-smooth sets, such as fractal attractors, without requiring addition of stochastic noise to regularize the dynamics. 

As a numerical application, we have studied how eigenfunctions of kernel integral operators utilizing delay-coordinate maps identify coherent observables of the L63 model---a system known to have a unique SRB measure with mixing dynamics \cite{Tucker99,LuzzattoEtAl05}, and thus absence of non-constant Koopman eigenfunctions in $L^2$. We found that for a sufficiently long embedding window (of approximately 8 Lyapunov times) the kernel-based approach, realized using a symmetric Markov kernel constructed by bistochastic normalization \cite{CoifmanHirn13} of a variable-bandwidth Gaussian kernel \cite{BerryHarlim16}, identifies through its two leading non-constant eigenfunctions an observable of the L63 system exhibiting a highly coherent dynamical behavior. This observable has an oscillatory period of approximately 0.76 natural time units, and remains coherent at least out to 10 natural time units (approximately 9 Lyapunov timescales) as measured by a 0.4 threshold of its time-autocorrelation function. Spatially, its real and imaginary parts have a structure that could be qualitatively described as a wavenumber 1 azimuthal oscillation about the holes in the two lobes of the attractor; a pattern that resembles observables previously identified through Koopman spectral analysis techniques appropriate for mixing dynamical systems \cite{DasEtAl20,KordaEtAl20}.       

Possible applied directions stemming from this work include detection of coherence in prototype models for metastable regime behavior in atmospheric dynamics \cite{CrommelinMajda04}, as well as PDE models with intermittency in both space and time \cite{MajdaEtAl97}. On the theoretical side, it would be interesting to explore connections between the spectral results presented here and geometrical characterizations of coherence, including the characterization given in DMDC based on the multiplicative ergodic theorem \cite{BerryEtAl13} and the dynamic isoperimetry approach proposed in \cite{Froyland15}. It may also be fruitful  to employ coherent eigenfunctions of integral operators based on delay-coordinate maps to construct approximation spaces for pointwise and/or spectral approximation of Koopman and transfer operators, including the extended dynamic mode decomposition (EDMD) technique \cite{WilliamsEtAl15} and the RKHS compactification approaches proposed in \cite{DasEtAl20}.

\begin{acknowledgements}
    The author is grateful to Andrew Majda for his guidance and mentorship during a postdoctoral position at the Courant Institute from 2009--2012. He is especially grateful for his friendship and collaboration over the years. This research was supported by NSF grant 1842538, NSF grant DMS 1854383, and ONR YIP grant N00014-16-1-2649.   
\end{acknowledgements}

\section*{Conflict of interest}
The author declares that he has no conflict of interest.

\appendix

\section{\label{appSpecConv}Proof of Proposition~\ref{propSpecConv}}

It is convenient to introduce an intermediate integral operator $ \mathcal K_{T,\Delta t} : C(M) \to C(M) $, 
\begin{displaymath}
    \mathcal K_{T,\Delta t} f = \int_\Omega k_{T,\Delta t}(\cdot, x)f(x) \, d\mu,
\end{displaymath}
which integrates against the invariant measure $\mu $ using the discrete-time delay-coordinate map, and split the analysis of the spectral convergence of $\mathcal K_{T,\Delta t, N} $ to $\mathcal K_T$ to two subproblems involving the convergence of (i) $\mathcal K_{T,\Delta t, N}$ to $ \mathcal K_{T,\Delta t}$ as $N \to \infty$; and (ii) $ \mathcal K_{T,\Delta t}$ to $ \mathcal K_T$ as $ \Delta t \to 0$. We now consider these two subproblems, starting from the second one.  

\paragraph{Spectral convergence of $ \mathcal K_{T,\Delta t} $ to $  \mathcal K_T$ as $ \Delta t \to 0$} The uniform convergence of the kernels $ k_{T,\Delta t}$ to $ k_T$, i.e., $ \lim_{\Delta t \to 0 } \lVert k_{T,\Delta t} - k_T \rVert_{C(M\times M)}$ (see~\eqref{eqDTConvergence}), implies convergence of $\mathcal K_{T,\Delta t}$ to $\mathcal K_T$ in $C(M)$ operator norm. It then follows from results on spectral theory of compact operators \cite{Atkinson67,Chatelin11} that the analogous claims to Proposition~\ref{propSpecConv} hold for the eigenvalues and spectral projections,  $ \lambda_{\Sigma,T,\Delta t} $ and $ \Pi_{\Sigma,T,\Delta t}$, respectively, of $\mathcal K_{T,\Delta t}$. That is, we have
\begin{equation}
    \lim_{\Delta t \to 0 } \lambda_{j,T,\Delta t} = \lambda_{j,T}, \quad \lim_{\Delta t \to 0 } \Pi_{\Sigma,T,\Delta t} f = \Pi_{j,T} f, \quad \forall f \in C(M), 
    \label{eqSpecConv1}
\end{equation}
where $ \Pi_{\Sigma, T, \Delta t} : C(M) \to C(M)$ is the spectral projection of $ \mathcal K_{T,\Delta t}$ onto $ \Sigma$.

\paragraph{Spectral convergence of $ \mathcal K_{T,\Delta t,N} $ to  $ \mathcal K_{T,\Delta t} $ as $  N \to 0$} Unlike the $ \mathcal K_{T,\Delta t} \to \mathcal K_T$ case, the operators  $ \mathcal K_{T,\Delta t, N}$ need not converge to $\mathcal K_{T,\Delta t}$ in $C(M)$ operator norm. In essence, this is because the weak convergence of measures in~\eqref{eqWeakConv} is not uniform with respect to $f$, even upon restriction to functions in $C(M)$. Nevertheless, as shown in \cite{VonLuxburgEtAl08}, the continuity of the kernel $k_T$ is sufficient to ensure that for a fixed $ f \in C(M)$, a restricted form of uniform convergence holds, namely 
\begin{equation}
    \label{eqGC}
    \lim_{N\to\infty} \sup_{g \in \mathcal G} \lvert \mathbb E_{\mu_N} g - \mathbb E_\mu g \rvert = 0,   
\end{equation}
where $\mathcal G \subset C(M)$ is the set of functions given by
\begin{displaymath}
    \mathcal G = \{ k_T(x,\cdot)f(\cdot) \mid x \in M \}.
\end{displaymath}
A collection of functions satisfying~\eqref{eqGC} is known as a Glivenko-Cantelli class. 

The Glivenko-Cantelli property turns out to be sufficient to ensure that as $N \to \infty$, the sequence of operators $\mathcal K_{T,\Delta t, N}$ exhibits a form of convergence to $ K_{T,\Delta t}$, called collectively compact convergence which, despite being weaker than norm convergence, is sufficiently strong to imply the spectral convergence claims in Proposition~\ref{propSpecConv}. We state the relevant definitions for collectively compact convergence below, and refer the reader to \cite{Chatelin11,VonLuxburgEtAl08} for additional details. 

\begin{definition} \label{defConv} Let $A_N : E \to E$ be a sequence of bounded linear operators on a Banach space $E$, indexed by $N \in \mathbb N$. 
    \begin{enumerate}[(i), wide]
        \item $A_N$ is said to converge to an operator $A : E \to E$ if $A_N $ converges to $A$ strongly, and for every uniformly bounded sequence $f_N \in E $ the sequence $(A - A_N) f_N$ has compact closure.   
        \item $ \{ A_N \} $ is said to be collectively compact if $ \cup_{N \in \mathbb N} A_n B$ has compact closure in $E$, where $B$ is the unit ball of $E$.   
        \item $A_N$ is said to converge to $A$ collectively compactly if it converges pointwise, and there exists $N_0 \in N$ such that for all $N > N_0$, $ \{ A_N \}_{N> N_0}$ is collectively compact.
    \end{enumerate}
\end{definition}

It can be shown that operator norm convergence implies collectively compact convergence, and collectively compact convergence implies compact convergence. The latter, is in turn sufficient for the following spectral convergence result:

\begin{lemma}
    \label{lemSpecConv}With the notation of Definition~\ref{defConv}, suppose that $A_N$ converges to $A$ compactly. Let $\lambda \in \sigma_p(A) $ be an isolated eigenvalue of $A$ with finite multiplicity $m$, and $\Sigma $ an open neighborhood of $ \lambda $ such that $ \sigma(A) \cap \Sigma = \{ \lambda \}$. Then, the following hold:
    \begin{enumerate}[(i), wide]
        \item There exists $N_0 \in \mathbb N $, such that for all $ N > N_0$, $ \sigma(A_N) \cap \Sigma$ is an isolated subset of the spectrum of $A_N$, containing at most $m$ distinct eigenvalues whose multiplicities sum to $ m $. Moreover, as $N \to \infty$, every element of $ \sigma(A_{N>N_0}) \cap \Sigma$ converges to $ \lambda $.
        \item As $N \to \infty $, the spectral projections of $A_N$ onto $ \sigma(A_N) \cap \Sigma $, defined in the sense of the holomorphic functional calculus, converge strongly to the spectral projection of $A_N$ onto $ \{ \lambda \} $.  
    \end{enumerate}
\end{lemma}

Using a similar approach as Proposition~13 in \cite{VonLuxburgEtAl08}, which employs, in particular, the Glivenko-Cantelli property in~\eqref{eqGC}, it can be shown that as $N \to \infty$, $ \mathcal K_{T,\Delta t, N}$ converges collectively compactly to $ \mathcal K_{T,\Delta t}$. Then, Lemma~\ref{lemSpecConv}, in conjunction with the fact that $\mathcal K_{T,\Delta t} $ is compact (so every nonzero element of its spectrum is an isolated eigenvalue of finite multiplicity), implies that 
\begin{equation}
    \lim_{N \to \infty } \lambda_{j,T,\Delta t,N} = \lambda_{j,T,\Delta t}, \quad \lim_{N \to \infty } \Pi_{\Sigma,T,\Delta t, N} f = \Pi_{\Sigma, T, \Delta t} f, \quad \forall \lambda_{j,T,\Delta t} \in \Sigma, \quad \forall f \in C(M), 
    \label{eqSpecConv2}
\end{equation}
where $\Sigma $ is the spectral neighborhood in the statement of the proposition. Proposition~\eqref{propSpecConv} is then proved by combining~\eqref{eqSpecConv1} and~\eqref{eqSpecConv2}. \qed

\bibliographystyle{spmpsci}      % mathematics and physical sciences
%\bibliography{./bib/bibliography}   % name your BibTeX data base

%\begin{thebibliography}{}
% etc
%\end{thebibliography}

\end{document}